\documentclass[12pt,reqno]{amsart}

\usepackage{latexsym}
\usepackage{amssymb}
\usepackage{graphicx}

\renewcommand{\Re}{{\operatorname{Re}\,}}
\renewcommand{\Im}{{\operatorname{Im}\,}}
\newcommand{\lla}{\left\langle}
\newcommand{\rra}{\right\rangle}

\newcommand{\s}{{\mathbf s}}

\renewcommand{\epsilon}{\varepsilon}

\newcommand{\Var}{{\bf Var}}

\newcommand{\sm}{\smallsetminus}
\newcommand{\szego}{Szeg\H{o} }

\newcommand{\inv}{^{-1}}
\newcommand{\kahler}{K\"ahler }
\newcommand{\sqrtn}{\sqrt{N}}
\newcommand{\wt}{\widetilde}
\newcommand{\wh}{\widehat}
\newcommand{\PP}{{\mathbb P}}
\newcommand{\N}{{\mathbb N}}
\newcommand{\R}{{\mathbb R}}
\newcommand{\C}{{\mathbb C}}

\newcommand{\CP}{\C\PP}
\renewcommand{\d}{\partial}
\newcommand{\dbar}{\bar\partial}
\newcommand{\ddbar}{\partial\dbar}
\newcommand{\U}{{\rm U}}

\newcommand{\E}{{\mathbf E}}

\newcommand{\half}{{\textstyle \frac 12}}
\newcommand{\vol}{{\operatorname{Vol}}}

\newcommand{\supp}{{\operatorname{Supp\,}}}
\renewcommand{\phi}{\varphi}
\newcommand{\eqd}{\buildrel {\operatorname{def}}\over =}

\newcommand{\acal}{\mathcal{A}}

\newcommand{\ccal}{\mathcal{C}}
\newcommand{\dcal}{\mathcal{D}}
\newcommand{\ecal}{\mathcal{E}}
\newcommand{\fcal}{\mathcal{F}}

\newcommand{\ical}{\mathcal{I}}

\newcommand{\lcal}{\mathcal{L}}

\newcommand{\ocal}{\mathcal{O}}
\newcommand{\pcal}{\mathcal{P}}

\newcommand{\scal}{\mathcal{S}}

\newcommand{\sss}{{\bf s}}
\newcommand{\al}{\alpha}
\newcommand{\be}{\beta}
\newcommand{\ga}{\gamma}

\newcommand{\La}{\Lambda}
\newcommand{\la}{\lambda}
\newcommand{\ep}{\varepsilon}
\newcommand{\de}{\delta}

\newcommand{\om}{\omega}
\newcommand{\Om}{\Omega}

\newtheorem{theo}{{\sc Theorem}}[section]
\newtheorem{maintheo}{{\sc Theorem}}

\newtheorem{maincor}[maintheo]{{\sc Corollary}}
\newtheorem{lem}[theo]{{\sc Lemma}}
\newtheorem{prop}[theo]{{\sc Proposition}}

\newenvironment{example}{\medskip\noindent{\it Example:\/} }{\medskip}

\newtheorem{defin}[theo]{{\sc Definition}}

\title[Random zeros on complex manifolds]
{Random zeros on complex manifolds: conditional expectations}

\author{Bernard Shiffman}
\author{Steve Zelditch}
\author{Qi Zhong}

\address{Department of Mathematics, Johns Hopkins University, Baltimore, MD
21218, USA} \email{shiffman@math.jhu.edu}

\address{Department of Mathematics, Northwestern University, Evanston, IL,
60208, USA}
\email{zelditch@math.northwestern.edu}

\address{Department of Mathematics, Vanderbilt University, Nashville, TN 37240, USA}
\email{qi.zhong@vanderbilt.edu}

\thanks{Research of the first author partially supported by NSF grant
DMS-0901333; research of the second author partially supported by
NSF grant     DMS-0904252.}

\begin{document}

\date{May 22, 2010}

\begin{abstract} We study the  conditional
 distribution $K^N_k(z | p)$ of zeros  of a Gaussian
system of random polynomials (and more generally, holomorphic sections), given that the polynomials or sections
vanish at a point $p$ (or a fixed finite set of points). The conditional distribution is analogous to the pair correlation function  of zeros but we show that it has quite a different small distance behavior. In particular, the conditional distribution does not exhibit repulsion of zeros in dimension one. To prove this, we
give universal scaling asymptotics for $K^N_k(z | p)$ around $p$.  The key tool is the conditional \szego kernel and its scaling asymptotics. \end{abstract}

\maketitle

\section{Introduction}

In this paper we study the  conditional expected distribution
 of zeros of a Gaussian random  system $\{s_1, \dots,
s_k\}$ of $k \leq m$  polynomials of degree $N$  in $m$ variables,
given that the polynomials $s_j$ vanish at a point $p\in M$, or at
a finite set of points $\{p_1, \dots, p_r\}$. More generally, we
consider systems of holomorphic sections of a degree $N$ positive
line bundle $L^N \to M_m $ over a compact \kahler manifold of
dimension $m$. The conditional expected distribution is the
current $ K_k^N(z |p)\in\dcal'^{k,k}(M)$ given by
\begin{equation}\label{CONDE}\Big( K_k^N(z | p), \phi \Big) : =  \E_N
\Big[(Z_{s_1, \dots, s_k}
,\phi)\Big|s_1(p)=\cdots=s_k(p)=0\Big],\quad
 \mbox{for }\ \phi\in\dcal^{m-k,m-k}(M)
\,. \end{equation}
Here,
$Z_{s_1, \dots, s_k}$ is the $(k, k)$ current of integration over
the simultaneous zeros of the sections; i.e., its pairing with a
smooth test form $\phi \in \dcal^{m - k, m -k}(M)$ is the integral $\int_{Z_{s_1, \dots, s_k}} \phi$
of the test form  over the joint zero set.  The
expectation $\E_N$ is the standard Gaussian conditional expectation  on
$\prod_1^kH^0(M, L^N)$,  which we condition on the linear random
 variable $(s_1,\dots,s_k) \mapsto (s_1(p),\dots,s_k(p))$  that evaluates the sections at the point $p$ (see
Definition \ref{defcondk}).

 We show that $K_k^N(z | p) $  is  a smooth
$(k, k)$ form away from $p$ (Lemma \ref{product}), and we
determine its asymptotics, both unscaled and scaled,  as $N \to
\infty$. Our main result, Theorem \ref{scaled} (for $k=m$) and Theorem \ref{all codim} (for $k<m$), is that the
scaling limit of $K_k^N(z |p)$ around the point $p$ is the
conditional expected distribution $K_{km}^\infty (z | 0)$ of joint zeros given a
zero at $z = 0$ in the Bargmann-Fock ensemble of entire
holomorphic functions on $\C^m$, and we  give an explicit formula
for $K_{km}^\infty (z | 0)$. Thus, the scaling limit is
universal.

Our  study of  $K_k^N(z | p) $ is parallel to our study of the
two-point correlation function $K_{2k}^N(z, p)$ for joint zeros in
our prior work with P. Bleher  \cite{BSZ,BSZ2}. There we showed
that $K_{2k}^N(z, p)$ similarly has a scaling limit given by the
pair correlation function $K^\infty _{2km}(z, 0)$ of zeros in  the
Bargmann-Fock ensemble.  Both $K_{km}^N(z |p) $ and $K_{2 km}^N(z,
p)$  measure a probability density of finding simultaneous zeros
at $z$ and at $p$: $K_{km}^N(z | p) $ is the result of
conditioning in a Gaussian space
 (see e.g. \cite{J}, Chapter 9.3), while $K_{2 km}^N(z, p)$ is a natural
 conditioning from the viewpoint of
random point processes (see \S \ref{PT}).  
Of special interest is the case $k=m$ where the joint zeros are (almost surely) points.
 In this case, the scaling limit (Bargmann-Fock)
conditional density $K_{mm}^\infty (z | 0)$ and  pair correlation
density $K_{2mm}^\infty (z, 0)$  turn out to have quite
different short distance behavior, as discussed in \S \ref{short} below.

To state our results, we need to recall the definition of a
Gaussian random system of holomorphic sections of a line bundle.
We let $(L, h) \to (M, \omega_h)$ be a positive Hermitian  holomorphic line
bundle over a compact complex manifold with \kahler  form $\omega_h=\frac
i2\Theta_h$. We then let $H^0(M, L^N)$ denote
the space of holomorphic sections of the $N$-th tensor power of
$L$. A special case is when $M  = \CP^m$, and $L = \ocal(1)$ (the
hyperplane section line bundle), in which case $H^0(\CP^m,
\ocal(N))$ is the space of  homogenous polynomials of degree $N$. As recalled
in $\S \ref{BACKGROUND}$, the Hermitian metric $h$ on $L$
induces inner products on $H^0(M,L^N)$ and these induce  a
Gaussian measure $\gamma_h^N$ on $H^0(M,L^N)$.  A Gaussian random
system is a choice of $k$ independent Gaussian random sections,
i.e. we endow $\prod_{j = 1}^k H^0(M, L^N)$ with the product
measure. We refer to $(\prod_{j = 1}^k H^0(M, L^N), \prod_{j=1}^k
\gamma_h^N)$ as the {\it Hermitian Gaussian ensemble} induced by
$h$. We let $\E_N=\E_{(\prod\gamma_h^N)}$ denote the expected value with
respect to $\prod
\gamma_h^N$.  Given $s_1, \dots, s_k \in H^0(M, L^N)$ we denote by $Z_{s_1,
\dots, s_k}$ the current of integration over the zero set $\{z \in
M: s_1(z) = \cdots = s_k(z) = 0\}$.  Further background is given
in \S \ref{BACKGROUND} and in \cite{SZ,BSZ,SZa}.

Our first result gives the asymptotics as $N \to \infty$ of the
conditional expectation of the zero current (\ref{CONDE}) of one
section. It shows that conditioning on $s(P) = 0$ only modifies
the unconditional zero current by a term of order $N^{-m}$, where $m=\dim M$.

\begin{maintheo}\label{unscaled} Let $(L, h) \to (M, \omega_h)$ be a positive
Hermitian  holomorphic line bundle over a compact complex manifold of dimension
$m$ with \kahler  form $\omega_h=\frac i2\Theta_h$, and let  $(H^0(M,L^N),
\gamma_{h}^N)$ be
the Hermitian Gaussian ensemble.  Let let $p_1,\dots, p_r$ be
distinct points of $M$. Then for all  test forms $\phi\in
\dcal^{m-1,m-1}(M)$, we have
$$ \E_N\Big[(Z_s,\phi)\Big|s(p_1)=\cdots=s(p_r)=0\Big]=\E_N(Z_s,\phi) -C_m
\,N^{-m}\sum_{j=1}^k\frac {i\ddbar \phi(p_j)}{\Om_M(p_j)}+O(N^{-m-1/2+\ep}),$$
 where $\Om_M= \frac 1{m!}\om_h^m$ is the volume form of $M$, and $C_m=\half
\pi^{m-1}\,\zeta(m+1)$.
\end{maintheo}

As mentioned above, the interesting problem is to rescale the
zeros around a fixed point $z_0$. When $k = m$ the joint zeros of
the system are almost surely a discrete set of points which are
$\frac{1}{\sqrt{N}}$-dense. Hence, we  rescale a
$\frac{C}{\sqrt{N}}$-ball around $z_0$ by $\sqrt{N}$ to make
scaled zeros  a unit apart on average from their nearest
neighbors. If $z_0 \not = p_j$ for any $j$, the scaled limit
density is just  the unconditioned scaled density, so we only
consider the case where $z_0 = p_{j_0}$ for some $j_0$. Then the
other conditioning points $p_j, j \not= j_0$, become irrelevant to
the leading order term, so we only consider the scaled conditional
expectation with  one conditioning point. Our main result is the following
scaling asymptotics

\begin{maintheo}\label{scaled}   Let $(L,h)\to (M,\om_h)$ and $(H^0(M,L^N),
\gamma_{h}^N)$ be as in Theorem \ref{unscaled}, and let $p\in M$. Choose normal
coordinates $z=(z_1,\dots,z_m):M_0,p\to \C^m,0$ on a neighborhood $M_0$ of $p$, and let
$\tau_N=\sqrtn\,z:M_0\to\C^m$ denote the scaled coordinate map.  

Let
$K_m^N(z|p)$ be the conditional expected zero distribution given by \eqref{CONDE} and Definition \ref{defcondk}.  Then for a
smooth test function
$\phi\in\dcal(\C^m)$, we have
\begin{multline*}\left(K^N_m(z|p)\,,\,\phi\circ \tau_N(z)\right)\\=\ \phi(0)\  +\ \int_{\C^m\sm
\{0\}} \phi(u)\left(
\frac i{2\pi}\ddbar\left[\log(1-e^{-|u|^2}) +|u|^2\right]\right)^m
\ +\ O(N^{-1/2+\epsilon})\,,\end{multline*} where
$u=(u_1,\dots,u_m)$ denotes the coordinates in $\C^m$.
\end{maintheo}

In \S \ref{proof2}, we give  a similar result (Theorem \ref{all codim}) for the conditional expected joint zero current $K_k^N(z|p)$ of joint zeros of codimension $k<m$.

Theorem \ref{scaled}  may be reformulated (without the remainder estimate) as
the following weak limit formula for currents:

\begin{maincor} \label{MAINCOR} Under the hypotheses and notation of Theorem
\ref{scaled},
\begin{eqnarray*}\tau_{N*}\left(K^N_m(z|p)\right)\ \to\ K_{mm}^\infty(u|0) & \eqd &  \de_0(u)+ \left(
\frac i{2\pi}\ddbar\left[\log(1-e^{-|u|^2}) +|u|^2\right]\right)^m
\\& =& \de_0(u)+ \frac
{1-(1+|u|^2)e^{-|u|^2}}{(1-e^{-|u|^2})^{m+1}}\left(\frac
i{2\pi}\ddbar|u|^2\right)^m\end{eqnarray*} weakly in
$\dcal'^{m,m}(\C^m)$, as $N\to\infty$.
\end{maincor}

The term $\delta_0(u)$ comes of course from the certainty of finding
a zero at $p$ given the condition. The form $\left(\frac
i{2\pi}\ddbar|u|^2\right)^m$ is the scaling limit of the
unconditioned distribution of zeros.

It follows from the proof that $K_{mm}^\infty(u|0)$ is the conditional density of common zeros
of $m$ independent random  functions in the  Bargmann-Fock ensemble of holomorphic functions on $\C^m$ of the form
  $$f(u) =  \sum_{J \in \N^m}\frac{c_{J}}{\sqrt{J!}}\,
u^J \;,$$ where the coefficients $c_{J}$ are
independent complex Gaussian random variables with mean 0 and
variance 1.
  The monomials $\frac{\pi^{-m/2}}{\sqrt{J!}}\,
u^J$ form a complete orthonormal basis of the Bargmann-Fock space of holomorphic functions that are in $L^2(\C^m, e^{- |z|^2} dz)$,
where $dz$ denotes Lebesgue measure. (We note that $f(u)$ is a.s.\ not in $L^2(\C^m, e^{- |z|^2} dz)$; instead, $f(u)$ is of finite order 2 in the sense of Nevanlinna theory. For further discussion of the Bargmann-Fock ensemble, see \cite {BSZ} and \S6 of the first version (arXiv:math/0608743v1) of \cite{SZa}.)

\subsection{Short distance behavior of the conditional density}
\label{short}

As in the case of the pair correlation function, Corollary
\ref{MAINCOR} determines the short distance behavior of the
conditional density of zeros around the conditioning point.

Before describing the results for the conditional density,  let us
recall the results  in \cite{BSZ,BSZ2}  for the pair correlation
function of zeros. The correlation function $K_{nk}^N(z_1, \dots,
z_n)$ is the probability density of finding  zeros of a system
of $k$ sections  at the $n$ points $z_1, \dots,
z_n$. For purposes of comparison to the conditional density, we
are interested  in the pair correlation density
$K_{2m}^N(z_1,z_2)$ for a full system of $k = m$ sections. It
gives the probability density of finding a pair of zeros of the
system at $(z_1, z_2)$. The scaling limit
\begin{equation}\label{slcd}\kappa_{mm}(|u|):=  \lim
_{N\to\infty}K^N_{1k}(p)^{-2} K_{2m}^N(p, p
+\frac{u}{\sqrtn})\,\end{equation} measures the asymptotic
probability of finding zeros at $p, p + \frac{u}{\sqrt{N}}$. As
the notation indicates, it depends only on the distance $r = |u|$
between the scaled points in the scaled metric around $p$. For
small values of $r$, it is proved in \cite{BSZ,BSZ2} that
\begin{equation} \label{leading} \kappa_{mm}(r)= \frac{m+1}{4}
r^{4-2m} + O(r^{8-2m})\,,\qquad\mbox{as }\ r\to 0\,.\end{equation}
This shows that the pair correlation function exhibits a striking
 dimensional dependence: When $m = 1, \kappa_{mm}(r) \to 0$ as $r \to 0$
and one has ``zero repulsion.'' When $m = 2$, $\kappa_{mm}(r) \to
3/4$ as $r \to 0$ and zeros neither repel nor attract.  With $m
\geq 3$, $\kappa_{mm}(r) \nearrow \infty$ as $r \to 0$ and there
joint zeros tend to cluster, i.e. it is   more likely to find a
zero at a small distance $r$ from another zero than at a small
distance $r$ from a given point.

The probability (density) of finding a pair of scaled zeros at
$(p, p + \frac{u}{\sqrt{N}})$ sounds similar to finding a second
zero at $p + \frac{u}{\sqrt{N}}$ if there is a zero at $p$, i.e.
the conditional probability density. Hence one might expect the
scaled conditional probability to resemble the scaled correlation
function. But Corollary \ref{MAINCOR} tells a different story. We
ignore the term $\delta_0$ (again) since it arises trivially from
the conditioning and only consider the behavior of the coefficient
\begin{equation} \label{leadingb}
\kappa_{m}^{\operatorname{cond}}(|u|) : = \frac {1-(1+|
u|^2)e^{-|u|^2}}{(1-e^{-|u|
^2})^{m+1}} \sim \frac{1}{2}\;|u|^{2 - 2 m} \end{equation} of the
scaling limit conditional distribution with respect to the Lebesgue
density $\left(\frac i{2\pi}\ddbar|u|^2\right)^m$ near $u = 0$.
The shift of the exponent down by $2$ in comparison to equation
(\ref{leading}) has the effect of shifting the dimensional
description down by one: In dimension one, the coefficient is
asymptotic to $\frac{1}{2}$  and therefore resembles the neutral
situation in our description of the pair correlation function.
Thus we do not see `repulsion' in the one dimensional conditional
density.  In dimension two, the conditional density
(\ref{leadingb}) is asymptotic to $\frac{1}{2} |u|^{-2}$, and
there is  a singularly enhanced probability of finding a zero near
$p$ similar to that for the pair correlation function in dimension
three; and so on in higher dimensions.

The following graphs illustrate the different behavior of these two  conditional zero distrbutions in low dimensions:
\begin{center}  \includegraphics[height=2in]{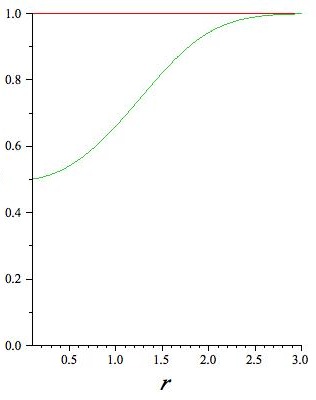}\hspace{1in}\includegraphics[height=2in]{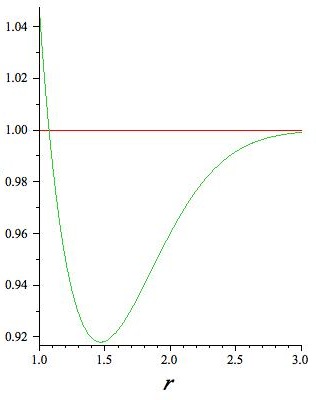}\\[-5pt]{\small$\kappa_{1}^{\operatorname{cond}}(r)
$\hspace{2in} $\kappa_{2}^{\operatorname{cond}}(r)$}
\\[16pt] \includegraphics[height=2in]{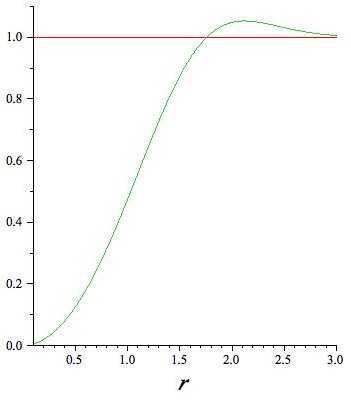}\quad 
\includegraphics[height=2in]{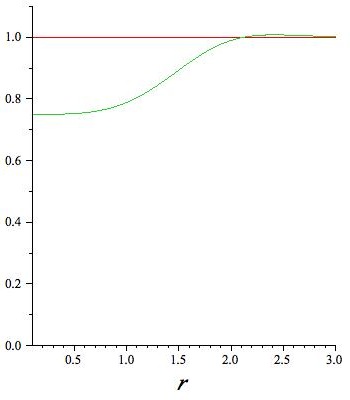}\quad \includegraphics[height=2in]{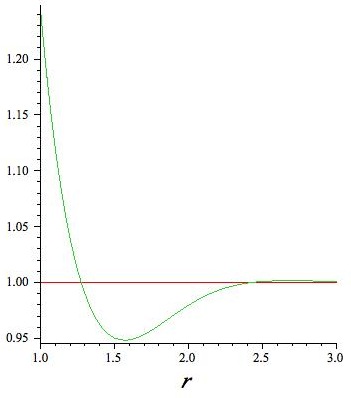}\\[-5pt]\hspace{.3in} {\small$\kappa_{11}^\infty(r)$\hspace{1.2in} $\kappa_{22}^\infty(r)$
\hspace{1.5in}$\kappa_{33}^\infty(r)$}\end{center}

 It is well known that conditioning on an
event of probability zero depends on the random variable used to
define the event. So there is no paradox, but possibly some
surprise, in the fact that the two conditional distributions are
so different.
See \S \ref{comparison} for further discussion of the comparison of the pair
correlation and the conditional density.

\section{\label{BACKGROUND} Background}
We begin with some notation and basic properties of sections of
holomorphic line bundles, Gaussian measures. The notation is the same as in
\cite{BSZ,SZ2,SZa}.

\subsection{Complex Geometry} We denote by $(L,h)\to M$ a Hermitian
holomorphic line bundle over a compact \kahler manifold $M$ of dimension $m$,
where $h$ is a smooth Hermitian metric with positive
curvature form
\begin{equation}
\Theta_h=-\partial\bar\partial\log \|e_L\|^2_h\,.
\end{equation}
 Here, $e_L$ is a local non-vanishing
holomorphic section of $L$ over an open set $U\subset M$, and
$\|e_L\|_h=h(e_L,e_L)^{1/2}$ is the $h$-norm of $e_L$. As in
\cite{SZa}, we give $M$ the Hermitian metric corresponding to the
K\"ahler form $\om_h=\frac{\sqrt{-1}}{2}\Theta_h$ and the induced
Riemannian volume form
\begin{equation}
\Om_M=\frac 1{m!}\,\om_h^m.
\end{equation}

  We denote by
$H^0(M,L^N)$ the space of holomorphic sections of
$L^N=L^{\otimes N}$. The metric $h$ induces Hermitian
metrics $h^N$ on $L^N$ given by $\|s^{\otimes N}\|_{h^N}=\|s\|^N_h.$
We give $H^0(M, L^N)$ the Hermitian inner product
\begin{equation}\label{inner product}
\langle s_1,s_2\rangle=\int_Mh^N(s_1,s_2)\,\Om_M  \ \ \ \ (s_1,s_2\in
H^0(M,L^N)),
\end{equation}
and we write $\|s\|=\langle s,s\rangle^{1/2}.$

 For a holomorphic section $s\in H^0(M, L^N)$, we let $Z_s\in\dcal'^{1,1}(M)$
denote the
current of integration over the zero divisor of $s$:
$$(Z_s, \phi)=\int_{Z_s}\phi, \ \ \ \phi\in \dcal^{m-1,m-1}(M),$$
where $\dcal^{m-1,m-1}(M)$ denotes the set of compactly supported $(m-1,m-1)$
forms on $M$.  (If $M$ has dimension 1, then $\phi$ is a compactly supported
smooth function.)
For $s=ge_L$ on an open set $U\subset M$, the
Poincar\'{e}-Lelong formula  states that
\begin{equation}\label{PL}
Z_s=\frac i\pi\partial\bar\partial\log
|g|=\frac i\pi \partial\bar\partial\log||s||_{h^N} +\frac N\pi\om_h.
\end{equation}

\subsubsection{\label{SZEGO} The Szeg\"o kernel} Let
$\Pi_N$: $L^2(M,L^N)\to H^0(M,L^N)$ denote the \szego projector with kernel $
\Pi_N$ given by
\begin{equation}
\Pi_N(z,w)=\sum^{d_N}_{j=1}S^N_j(z)\otimes \overline{S^N_j(w)}\in  L^N_z\otimes
\overline L^N_w\,,
\end{equation}
where
$\{S^N_j\}_{1\le j\le d_N}$ is an orthonomal basis of $H^0(M,L^N)$.

We shall use the  {\it normalized \szego
kernel}
\begin{equation}\label{PN} P_N(z,w):=
\frac{\|\Pi_N(z,w)\|_{h^N}}{\|\Pi_N(z,z)\|_{h^N}^{1/2}\,
\|\Pi_N(w,w)\|_{h^N}^{1/2}}\;.\end{equation}
(Note that $\|\Pi_N(z,w)\|_{h^N} = \sum\|S^N_j(z)\|_{h^N(z)}\,\|S^N_j(w)\|
_{h^N(w)}$, which equals the absolute value of the \szego kernel lifted to the
associated circle bundle,  as described in \cite{SZ2,SZa}.)

We have the $\ccal^\infty$ diagonal asymptotics  for the \szego kernel
(\cite{C,Z}):
\begin{equation}\label{Zel} \|\Pi_N(z,z)\|_{h^N} = \frac {N^m}{\pi^m} +
O(N^{m-1})\;.\end{equation}
Off-diagonal estimates for the
normalized \szego kernel $P_N$ were given in \cite{SZa}, using the off-diagonal
asymptotics for $\Pi_N$ from
\cite{BSZ,SZ2}. These estimates are of two types:\\[5pt]
1) `far-off-diagonal' asymptotics (Proposition 2.6 in \cite{SZa}):
For
$b>\sqrt{j+2k}$, $j,k\ge 0$, we have
\begin{equation}\label{far} \nabla^j
P_N(z,w)=O(N^{-k})\qquad \mbox{uniformly for }\ d(z,w)\ge
b\,\sqrt{\frac {\log N}{N}} \;.\end{equation}
(Here, $\nabla^j$ stands for the $j$-th covariant derivative.)\\[5pt]
2) `near-diagonal' asymptotics (Propositions 2.7--2.8 in \cite{SZa}):
Let $ z_0\in M$. For $\ep,b>0$,
there are constants $C_j=C_j({M,\ep,b})$, $j \ge 2$, independent
of the point $z_0$, such that
\begin{equation}\label{near}\textstyle  P_N\left(z_0+\frac u{\sqrtn},z_0 +\frac
v{\sqrtn}\right) =
e^{-\frac 12 |u-v|^2}[1 + R_N(u,v)]\;,\end{equation} where
\begin{equation}\label{nearr} \begin{array}{c}|R_N(u,v)|\le \frac {C_2}2\,|u-v|
^2N^{-1/2+\ep}\,, \quad
|\nabla R_N(u)| \le C_2\,|u-v|\,N^{-1/2+\ep}\,,
\\[8pt] |\nabla^jR_N(u,v)|\le C_j\,N^{-1/2+\ep}\quad j\ge 2\,,\end{array}
\end{equation}
for $|u|+|v|<b\sqrt{\log N}$.  (Here, $u, v$ are normal coordinates near
$z_0$.)

The limit on the right side of (\ref{near}) is the normalized
\szego kernel for the Bargmann-Fock ensemble (see \cite{BSZ}).
This is why the scaling limits of the correlation functions and
conditional densities coincide with those of the Bargmann-Fock
ensemble.

\subsection{Probability}  If $V$ is a finite dimensional complex vector space,  we shall associate a complex Gaussian probability measure $\ga$ to each Hermitian inner product on $V$ as follows:  Choose an orthonormal basis  $v_1,\dots,v_n$ for the inner product and  define $\ga$ by
\begin{equation}
d\gamma(v)=\frac{1}{\pi^{n}}e^{-|a|^2}d_{2n}a, \ \
s=\sum^{n}_{j=1}a_jv_j\in V\,,
\end{equation} where  $d_{2n}a$ denotes $2n$-dimensional Lebesgue measure.
This Gaussian
is characterized by the property that the $2n$ real variables
$\Re a_j$, $\Im a_j$ ($j=0,....,d_N$) are independent random
variables with mean 0 and variance $\frac12$; i.e.,
$$\mathbf{E}_\ga a_j=0, \ \ \mathbf{E}_\ga a_ja_k=0, \ \ \mathbf{E}_\ga a_j\bar a_k=
\delta_{jk}\,.$$ 
Here and throughout this article, $\mathbf{E}_\ga$ denotes expectation with respect to the probability measure $\ga$:
$\mathbf{E}_\ga\phi=\int\phi \,d\gamma$.
Clearly, $\ga$ does not depend on the choice of orthonormal basis, and each (nondegenerate) complex Gaussian measure on $V$ is associated with a unique (positive definite) Hermitian inner product on $V$.

In particular, we give $H^0(M,L^N)$ the complex Gaussian probability measure $\ga_h$
induced by the inner product \eqref{inner product}; i.e., 
\begin{equation}
d\gamma_h(s)=\frac{1}{\pi^{d_N+1}}e^{-|a|^2}da, \ \
s=\sum^{d_N}_{j=1}a_jS^N_j,
\end{equation}
where $\{S^N_j:1\leq j\leq d_N\}$ is an orthonormal basis for
$H^0(M,L^N)$ with respect to (\ref{inner product}).
The probability space $(H^0(M,L^N),\ga_N)$ is called the {\it Hermitian Gaussian ensemble\/}.
We  regard the currents $Z_s$ (resp.\ measures $|Z_s|$), as
current-valued (resp.\ measure-valued) random variables on $(H^0(M,L^N),\ga_N)$; i.e., for each test form
(resp.\ function) $\phi$, ($Z_s,\phi$) (resp. ($|Z_s|,\phi$)) is a
complex-valued random variable.

Since the zero current $Z_s$ is unchanged when $s$ is multiplied
by an element of $\C^*$, our results remain the same if we instead
regard $Z_s$ as a random variable on the unit sphere $SH^0(M,L^N)$
with Haar probability measure. We prefer to use Gaussian measures
in order to facilitate computations.

\subsubsection{Holomorphic Gaussian random fields} Gaussian random fields are
determined by their two-point
functions or covariance functions.  We are mainly interested in the
case where the fields are holomorphic sections of $L^N$; i.e, our probability
space is a subspace
$\scal$  of the space $ H^0(M, L^N)$ of holomorphic sections of $L^N$ and the
probability measure on $\scal$ is the Gaussian measure induced by the inner
product \eqref{inner product}.  If we pick an
orthonormal basis $\{S_j\}_{1\le j\le m}$ of $\scal$ with respect to
(\ref{inner product}), then we may write $s
=\sum^{n}_{j=0}a_jS_j $, where the coordinates $a_j$ are i.i.d.\ complex
Gaussian random variables. The two point function
\begin{equation}
\Pi_\scal(z,w):=\E_\scal\left(s(z)\otimes\overline{s(w)}\right)=
\sum^{n}_{j=1}S_j(z)\otimes \overline{S_j(w)}\end{equation}
is the kernel of the  orthogonal projection onto $\scal$, and equals the \szego
kernel $\Pi_N(z,w)$ when $\scal = H^0(M, L^N)$.  The expected zero current $\E_
\scal\big(Z_s\big)$ for random sections $s\in\scal$ is given by the {\it
probabilistic Poincar\'e-Lelong formula:}

\begin{lem}\label{important lemma} {\rm \cite{SZ}} Let $(L,h)\to M$ be a
Hermitian holomorphic line bundle over a compact complex manifold $M$ and let $
\scal\subset H^0(M,L^N)$ be a Gaussian random field with two-point function $
\Pi_\scal(z,w)$. Then
$$\E_\scal\big(Z_s\big) = \frac i{2\pi} \ddbar\log\|\Pi_\scal(z,z)\|_{h^N} +
\frac N{2\pi}\,\sqrt{-1}\,\Theta_h\,.$$
\end{lem}

This lemma was given in \cite[Prop.~3.1]{SZ} and \cite[Prop.~2.1]{SZa} with
slightly different hypotheses.  For convenience, we include a proof below.

\begin{proof}
Let $\{S_j\}_{1\le j\le n}$ be a basis of $\scal$ such that  $s\in\scal$ is of
the form $s=\sum_{j=1}^n a_j S_j$, where the $a_j$ are independent standard
complex Gaussian random variables, as above.  We then have $\|\Pi_\scal(z,z)\|
_{h^N}= \sum_{j=1}^n\|S_j\|^2_{h^N}$. For any
$s\in \scal$, we write $$s=\sum_{j=1}^n a_jS_j=\langle a,F\rangle e_L^{\otimes
N},$$ where $e_L$ is a local
non-vanishing holomorphic section of $L$,  $S_j=f_j\,e_L^{\otimes N}$,
 and $F=(f_1,...,f_n).$ We then
write $F(z)=|F(z)|U(z)$ so that $|U(z)|\equiv 1$ and
$$\log|\langle a,F\rangle |=\log|F|+\log|\langle a,U\rangle|.$$
A key point is that $\E(\log|\langle a,U\rangle|)$ is independent of $z$, and
hence $\E(d\log|\langle a,U\rangle|)=0$.  We note that $U$ is well-defined a.e.
on $ M\times \scal$; namely, it is defined whenever $s(z)\neq 0$.

Write $d\ga= \frac 1{\pi^n}e^{-|a|^2}\,da$.
By \eqref{PL}, we have
\begin{align*}
(\E Z_s,\phi)&= \E\left(\frac{\sqrt{-1}}{\pi} \ddbar\log|\langle a,F\rangle |,
\phi\right)\
                  =\ \frac{\sqrt{-1}}{\pi}\int_{\C^{n}}(\log|\langle a,F\rangle
|,\partial\bar\partial
                 \phi)\,d\gamma\\
                 &=\frac{\sqrt{-1}}{\pi}\int_{\C^{d_N}}(\log|F|,\partial\bar
\partial
                 \phi)\,d\gamma+\frac{\sqrt{-1}}{N}\int_{\C^n}(\log|\langle a,U
\rangle|,\partial\bar\partial\phi)\,d\gamma,
\end{align*}
for all test forms $\phi\in \dcal^{m-1,m-1}(M)$. The first term is
independent of $a$, so we may remove the Gaussian integral. The
vanishing of the second term follows by noting that
\begin{align*}
\int_{\C^{n}}(\log|\langle a,U\rangle|,\partial\bar\partial\phi)\,d\gamma&=
\int_{\C^{n}}\,d\gamma\int_M\log|\langle a,U\rangle|\,\partial\bar\partial\phi\
\
&=\int_M\int_{\C^{n}}\log|\langle a,U\rangle|\,d\gamma\,\partial\bar\partial
\phi=0,
\end{align*}
since
$\int\log|\langle a,U\rangle|\,d\gamma=\frac{1}{\pi}\int_{\C}\log|a_0|e^{-|a_0|
^2}\,da_0$
is constant, by the $\U(n)$-invariance of $d\gamma$.
Fubini's Theorem can be applied above since
$$\int_{M\times\C^{n}}\big|\log|\langle a,U\rangle|\,\partial\bar\partial\phi
\big|\,d\gamma=\frac{1}{\pi}\int_{\C}\big|\log|a_0|\,\big|e^{-|a_0|^2}\,da_0\,
\int_M
|\partial\bar\partial\phi|<+\infty.$$
Thus \begin{eqnarray*}
\E Z_s &=&\frac{\sqrt{-1}}{2\pi}\partial\bar\partial\log|F|^2\ =\
\frac{\sqrt{-1}}{2\pi}\partial\bar\partial\left(\log\sum^{n}_{j=1}\|S_j\|^2_h
-\log\|e_L\|^2_h\right)\\&=&  \frac {\sqrt{-1}}{2\pi} \ddbar\log\|\Pi_
\scal(z,z)\|_h +\frac {\sqrt{-1}}{2\pi}\Theta_h\,.
\end{eqnarray*}
 \end{proof}

\section{Conditioning on the values of a random variable}\label{conditioning}

In this section, we give a precise definition of the conditional
expected zero current
$\E\big(Z_{s_1,\dots,s_k}\big|s_1(p)=v_1,\,\dots,\,s_k(p)=v_k\big)$
(Definition \ref{defcondk}) and give a number of its properties.
In particular, we give a formula for
$\E\big(Z_s\big|s(p_1)=\cdots=s(p_r)=0\big)$ in terms of the
conditional \szego kernel (Lemma \ref{linear}).

\subsection{The Leray form}\label{LERAY}
 We first give a general formula for the
conditional expectation $\E(X|Y=y)$ of a continuous random
variable $X$ with respect to a smooth random variable $Y$ when $y$
is a regular value of $Y$. Our discussion differs from the
standard expositions, which do not tend to assume random variables
to be smooth.

We begin by recalling the definition of the  conditional
expectations $\E(X|\fcal)$ of a random variable $X$ on a
probability space $(\Om,\acal,P)$ given a sub-$\sigma$-algebra
$\fcal\subset\acal$:

\begin{defin} Let $X$ be a random
variable $X$ with finite first moment (i.e., $X\in L^1$) on a probability space
$(\Om,\acal,P)$, and let $\fcal\subset\acal$ be a $\sigma$-algebra. The
conditional expectation  is a random variable $E(X|\fcal)\in L^1(\Om,P)$
satisfying:
\begin{itemize}
\item $\E(X|\fcal)$ is measurable with respect to $\fcal$;
\item For all sets $A\in\fcal$, $\ \int_A \E(X|\fcal)\,dP =\int_A X\,dP$.

\end{itemize}
\end{defin}
The existence and uniqueness (in $L^1$)  of $E(X|\fcal)$ is a standard fact
(e.g.,  \cite[Th.~6.1]{K}).

In this paper, we are interested in the conditional expectation
$\E(X|\sigma(Y ))$ of a continuous random variable $X$ on a
manifold $\Om$ with respect to a smooth  random variable
$Y:\Om\to\R^k$. Here, $\sigma(Y )$ denotes the $\sigma$-algebra
generated by $Y$, i.e.\ the pull-backs by $Y$ of the Borel sets in
$\R^k$;
 $\sigma(Y )$ is  generated by the sublevel sets $\{Y_j\leq t_j , j
= 1, \dots , k\}$. The condition that  $\E(X|\sigma(Y ))$ is measurable with
respect to $\sigma(Y)$ implies that it is constant on the level sets of $Y$.
We then write
$$ \E(X|Y=y)\ :=\ \E(X|\sigma(Y ))(x)\,,\quad x\in Y\inv(y)\,.$$
We call $ \E(X|Y=y)$ the {\it conditional expectation of $X$ given that $Y=y$.}
We note that the function $y\mapsto \E(X|Y=y)$ is in $L^1(\R^k,Y_*P)$, and is
not necessarily well-defined at each point $y$.  However, in the cases of
interest to us, $ \E(X|Y=y)$ will be a continuous function.

To give a geometrical description of $\E(X|Y )$, we use the language of
Gelfand-Leray forms:

\begin{defin} Let $Y : \Om \to \R^k$ be a $\ccal^\infty$ submersion where $\Om$
is an
oriented $n$-dimensional manifold. Let $\nu\in \ecal^n$, e.g. a
volume form. The Gelfand-Leray form $\lcal(\nu,Y,y ) \in \ecal^{n-k}(Y^{-1}(y))
$ on the
level set $\{Y = y\}$ is given by
\begin{equation}
\lcal(\nu,Y,y)\wedge dY_1 \wedge\cdots\wedge dY_k = \nu \ \ \mbox{on }\ {Y^{-1}
(y)}\,,\ \ \ i.e.,\  \lcal(\nu,Y,y) = \frac{\nu}{dY_1\wedge\cdots\wedge dY_k
}\bigg|_{Y\inv(y)} .
\end{equation}
\end{defin}

Conditional expectation of a random variable is a form of
averaging. The following Proposition shows this explicitly: it
amounts to averaging $X$ over the level sets of $Y$.

\begin{prop} \label{Leray} Let $\nu\in\ecal^n(\Om)$ be a smooth probability
measure on a manifold $\Om$. Let $Y : \Om \to \R^k$ be a $\ccal^\infty$
submersion, and let $X\in L^1(\Om,\nu)$. Then
$$\E(X|Y=y) = \frac{\int_{Y =y} X \,\lcal(\nu,Y,y ) } {\int_{Y =y} \lcal(\nu,
Y,y)}.$$
\end{prop}

\begin{proof}We first note that
$$ \int_{y\in \R^k}\left(\int_{Y^{-1}(y)}|X|\,\lcal(\nu,Y,y )\right)dy_1\cdots
dy_k=\int_X|X|\,\nu=1\,,$$
and hence $\int_{Y^{-1}(y)}|X|\,\lcal(\nu,Y,y )<+\infty$ for almost all $y\in
\R^k$. Furthermore  $\int_{Y^{-1}(y)}\lcal(\nu,Y,y )>0$ for $Y_*\nu$-almost all
$y\in\R^k$, and therefore $\E(X|Y=y)$ is well defined for $Y_*\nu$-almost all
$y$.
 Now let $$\wt E(x) =  \frac{\int_{Y =Y(x)} X \,\lcal(\nu,Y,Y(x) ) } {\int_{Y
=Y(x)} \lcal(\nu, Y,Y(x))}\,, \quad \mbox{for\ } \nu\mbox{-almost all\ }\  x\in
X.$$ The function $\wt E$ is measurable with respect to
$\sigma(Y )$ since it is the pull-back by $Y$ of a measurable function on $\R^k
$.

The only other thing to check is that $\int_A \wt E\,\nu =
\int_A X\,\nu$ for all $A\in\fcal$. It suffices to check this for
sets $A$ of the form $Y^{-1}(R)$ where $R$ is a rectangle in $\R^k$. But
then by the change of variables formula and Fubini's theorem,
\begin{multline*}\int_{Y^{-1}(R)}\wt E\,\nu = \int_{y\in R}\left(\int_{Y^{-1}
(y)}\wt E\,\lcal(\nu,Y,y )\right)dy_1\cdots dy_k\\= \int_{y\in R}
\left(\int_{Y^{-1}(y)}X\,\lcal(\nu,Y,y )\right)dy_1\cdots dy_k = \int_{Y^{-1}
(R)}
 X\,d\nu.\end{multline*} By uniqueness of the conditional expectation, we then
conclude that $\wt E = \E(X|\sigma(Y ))$.
\end{proof}

\begin{example} Let $\Om=\C^n$ with Gaussian probability measure $d\ga_n=\pi^{-
n}e^{-|a|^2}\,da$.
 Let $\pi_k:\C^n\to\C^k$ be the projection $
\pi_k(a_1,\dots,a_n)=(a_1,\dots,a_k)$.  For $y\in\C^k$ we have
$$\lcal(d\ga_n,\pi_k,y)=\frac 1{\pi^k}\,e^{-(|y_1|^2+\cdots+|y_k|^2)}\,d\ga_{n-
k}(a_{k+1},\dots,a_n)\,,$$
where  $$d\ga_{n-k}(a_{k+1},\dots,a_n)=e^{-(|a_{k+1}|^2+\cdots + |a_n|^2)}\,
\left(\frac i{2\pi}\right)^{n-k} da_{k+1}\wedge d \bar a_{k+1}\wedge\cdots
\wedge  da_n\wedge d \bar a_n$$
is the standard complex Gaussian measure on $\C^{n-k}$. For a  bounded random
variable $X$ on $\C^n$, let $X_y$ be
 the random variable on $\C^{n-k}$ given by $X_y(a')=X(y,a')$ for $a'\in\C^{n-
k}$. By Proposition \ref{Leray}, we then have
\begin{equation}\label{condgauss} \E_{\ga_n}(X|\pi_k=y)\ =\ \E_{\ga_{n-k}}(X_y)
\,.\end{equation}
\end{example}

This example leads us to the following definition:

\begin{defin}\label{defcg}  Let $\ga$ be a complex Gaussian measure on a finite dimensional complex space $V$, and let $W$ be a subspace of $V$. We define the {\em conditional Gaussian measure} $\ga_W$ on $W$ to be the Gaussian measure associated with the Hermitian inner product on $W$ induced by the inner product on $V$ associated with $\ga$. \end{defin}

The terminology of Definition \ref{defcg} is justified by the following proposition, which we shall use to define the expected zero current
conditioned on the value of a random holomorphic section at a point or points:

\begin{prop} \label{special}Let  $T:\C^n\to V$ be a linear map onto a complex
vector space $V$. Let $E$ be a closed subset of $\C^n$ such that $E\cap T
\inv(y)$ has  Lebesgue measure $0$ in $T\inv(y)$ for all $y\in V$. Let $X$ be a
bounded random variable on $\C^n$ such that $X|(\C^n\sm E)$ is continuous. Then
$\E_{\ga_n}(X|T=y)$ is continuous on $\C^k$.  Furthermore
$$\E_{\ga_n}(X|T=0)\ =\ \E_{\ga_{\ker T}}(X')\,,$$ where $X'$ is the restriction of $X$ to $\ker T$ and $\ga_{\ker T}$ is the conditional Gaussian measure on $\ker T$ as defined above.
\end{prop}

\begin{proof} Let $k=\dim V$.  We can assume without loss of generality that $
\ker T= \{0\}\times \C^{n-k}$.  Then the map $T$ has the same fibers as the
projection $\pi_k(a_1,\dots,a_n)=(a_1,\dots,a_k)$, and thus $\sigma(T)=
\sigma(\pi_k)$. Hence we can assume without loss of generality that $V=\C^k$
and $T=\pi_k$.

Fix $y_0\in\C^k$ and let $\ep>0$ be arbitrary. Choose a compact set $K\subset
\C^{n-k}$ such that
$(\{y_0\}\times K)\cap E=\emptyset$ and $\ga_{n-k}(\C^{n-k}\sm K) <\ep/\sup |X|
$. Since $E$ is closed, $(\{y\}\times K)\cap E=\emptyset$, for $y$ sufficiently
close to $y_0$. As above, we let $X_y(a')=X(y,a')$ for $a'\in\C^{n-k}$. Since
$X_y\to X_{y_0}$ uniformly on $K$,
we have
\begin{equation}\label{l1}\lim_{y\to y_0}\int_K X_y\,d\ga_{n-k} =
\int_KX_{y_0}\,d\ga_{n-k}\;.\end{equation}

It follows from \eqref{condgauss} that
\begin{equation}\label{l2} \left|\E_{\ga_n}(X|\pi_k=y) - \int_KX_y\,d\ga_{n-k}
\right|
= \left|\int_{\C^{n-k}\sm K}X_y\,d\ga_{n-k}\right|<\ep\,,\end{equation}
for all $y\in\C^k$.  The first conclusion is an immediate consequence of \eqref{l1}--
\eqref{l2} and the formula for $\E_{\ga_n}(X|T=0)$ follows from \eqref{condgauss} with $y=0$.
\end{proof}

\subsection{Conditioning on the values of sections}  We now state precisely
what is meant by the expected zeros conditioned on  sections having  specific
values at one or several points on the manifold:

\begin{defin}\label{defconditioning}  Let $(L,h)$ be a positive Hermitian
holomorphic line bundle over a compact \kahler manifold $M$ with \kahler form $
\om_h$.  Let $p_1,\dots,p_r$ be distinct points of $M$. Let  $N\gg 0$ and give
$H^0(M,L^N)$ the induced Hermitian Gaussian measure $\ga_N$. Let $v_j\in
L^N_{p_j}$, for $1\le j\le r$. We let $$T: H^0(M,L^N)\to L^N_{p_1}\oplus\cdots
\oplus L^N_{p_r}\,,\qquad s\mapsto s(p_1)\oplus\cdots\oplus s(p_r)\,.$$
The  expected zero current $\E\big(Z_s\big|s(p_1)=v_1,\,\dots,\,s(p_r)=v_r\big)
$ conditioned on the section taking the fixed values $v_j$ at the points $p_j$
is defined by:
$$\bigg(\E_N\big(Z_s\big|s(p_1)=v_1,\,\dots,\,s(p_r)=v_r\big)\,,\,\phi\bigg)\ =\
\E_{\ga_N}\big((Z_s,\phi)\big|T=
v_1\oplus\cdots\oplus v_r\big)\,,$$ for smooth test forms $\phi\in
\dcal^{m-1,m-1}(M)$.

\end{defin}

\begin{lem}\label{continuity1}
The  mapping $$ v_1\oplus\cdots\oplus v_r \mapsto \E_N\big(Z_s\big|s(p_1)=v_1,\,
\dots,\,s(p_r)=v_r\big)$$ is a continuous map  from $L^N_{p_1}\oplus\cdots
\oplus L^N_{p_r}$ to $ \dcal'^{1,1}(M)$. \end{lem}

\begin{proof} Let $N$ be sufficiently large so that $T$ is surjective.  Let  $
\phi\in \dcal^{m-1,m-1}(M)$ be a smooth test form, and consider the random
variable $X(s)=(Z_s,\phi)$ on $H^0(M,L^N)\sm \{0\}$. By \cite[Th.~3.8]{St}
applied to the projection $$\{(s,z)\in H^0(M,L^N)\times M: s(z)=0\}\to
H^0(M,L^N)\,,$$ the random variable $X$ is continuous on $H^0(M,L^N)\sm \{0\}$.
Furthermore, $X$ is bounded, since we have  by \eqref{PL}, $$|X(s)| \le (\sup
\|\phi\|) \,(Z_s,\om^{m-1}) = \frac N\pi (\sup \|\phi\|) \,\int_M\om_h^m\,,$$
The conclusion follows from
Proposition \ref{special} with $E=\{0\}$.\end{proof}

We could just as well condition on the section having specific derivatives, or
specific $k$-jets, at specific points.  At the end of this section, we discuss
the conditional zero currents of simultaneous sections.

We are particularly interested in the case where the $v_j$ all vanish. In this
case, the conditional expected current $\E_N\big(Z_s\big|s(p_1)=
\cdots=s(p_r)=0\big)$ is well-defined and we have:

\begin{lem}\label{linear} Let $(L,h)\to (M,\om_h)$ and  $(H^0(M,L^N),
\gamma_{h})$
 be as in Theorem \ref{scaled}. Let $p_1,\dots,p_r$ be distinct points of $M$
and let $H_N^{p_1\cdots p_r}\subset H^0(M,L^N)$ denote the space of holomorphic
sections of $L^N$ vanishing at  the points $p_1,\dots,p_r$. Then
$$\E_N\big(Z_s\big|s(p_1)=\cdots=s(p_r)=0\big)\ =\ \E_{\ga_N^{p_1\cdots p_r}}(Z_s)\ =\ \frac i{2\pi} \ddbar\log\|
\Pi_N^{p_1\cdots p_r}(z,z)\|_{h^N} +\frac N{\pi}\,\om_h\,,$$ where $\ga_N^{p_1\cdots p_r}$ is the conditional Gaussian measure on $H_N^{p_1\cdots p_r}$, and 
$\Pi_N^{p_1\cdots p_r}$ is the \szego kernel for the orthogonal projection onto
$H_N^{p_1\cdots p_r}$.
\end{lem}
\begin{proof} Let $\phi\in \dcal^{m-1,m-1}(M)$ be a smooth test form.  By
Proposition \ref{special},
$$\bigg(\E_N\big(Z_s\big|s(p_1)=\cdots=s(p_r)=0\big)\,,\,\phi\bigg)\ =\  \E_N
\big((Z_s,\phi)\big|T=
0\big)\ =\ \E_{\ga_N^{p_1\cdots p_r}}(Z_s,\phi)\,,$$ where $T$ is as in Definition
\ref{defconditioning}.
By Lemma \ref{important lemma} with $\scal = H_N^{p_1\cdots p_r}$, we then have
\begin{equation*}\E_{\ga_N^{p_1\cdots p_r}}(Z_s,\phi) =\ \left(\frac i{2\pi}
\ddbar\log\|\Pi_N^{p_1\cdots p_r}(z,z)\|_{h^N} +\frac N{\pi}\,\om_h\,,\,\phi
\right).\end{equation*}
\end{proof}

Recalling the definition of $P_N$ from (\ref{PN}), we now prove:
\begin{prop} \label{cond} We have  \begin{align}
\E_N\big(Z_s\big|s(p)=0\big) &=\E_N\big(Z_s\big)+\frac i{2\pi} \ddbar
\log \left( 1 -
            P_N(z,p)^2\right)\,,\end{align}
            \end{prop}

\begin{proof} As above, we let $H^p_N\subset
H^0(M, L^N)$ denote the space of  holomorphic sections vanishing
at $p$. Let $\{S^p_{Nj}:j=1,...,d_N-1\}$ be an orthonormal basis of $H^p_N$.
The Szeg\"o projection
$\Pi_N^p$ is given by
$$\Pi^p_N(z,w)=\sum S^p_{Nj}(z)\otimes \overline{S^p_{Nj}(w)}\,.$$
By Lemma \ref{linear} with $r=1$, we have
\begin{equation} \label{conditionalPL} \E_N\big(Z_s\big|s(p)=0\big) = \frac
i{2\pi} \ddbar\log\|\Pi_N^p(z,z)\|_{h^N} +\frac N\pi \om_h\,.\end{equation}

To give a formula for $\Pi_N^p(z,z)$, we consider the {\it
coherent state\/} at $p$, $\Phi^p_N(z)$ defined as follows: Let
\begin{equation}\label{coherent}\wh\Phi^p_N(z):=\frac{\Pi_N(z,p)}{\|\Pi_N(p,p)
\|_{h^N}^{1/2}} \in H^0(M,L^N)\otimes\overline L^N_p\,,\end{equation}
We choose a unit vector $e_p\in L_p$, and we let $\Phi^p_N\in
H^0(M,L^N)$ be given by
\begin{equation}\label{coherent1}\wh\Phi^p_N(z) = \Phi^p_N(z)  \otimes
\overline{e^{\otimes N}_p}\,.\end{equation}

The coherent state $\Phi^p_N$ is orthogonal to $H^p_N$, because
\begin{equation}\label{orthog}s\in H^p_N \ \implies\ \|\Pi_N(p,p)\|_{h^N}^{1/2}
\lla s,\wh\Phi_N^p\rra=\int_M\Pi_N(p,z)\,s(z)\,\Om_M(z)=s(p)=0
\end{equation}
Furthermore, $\|\Phi_N^p\|_{h^N}^2=1$, and hence
$\{S^p_{Nj}:j=1,...,d_N-1\}\cup\{\Phi^p_N\}$ forms an orthonormal
basis for $H^0(M,L^N)$. Therefore
\begin{equation}\label{cond szego}\Pi^p_N(z,w)=\Pi_N(z,w)-\Phi^p_{N}(z)\otimes
\overline{\Phi^p_{N}(w)}\,,\end{equation}
and in particular
\begin{equation} \label{condszegodiag}\|\Pi^p_N(z,z)\|_{h^N}= \|\Pi_N(z,z)\|
_{h^N}- \|\Phi^p_{N}(z)\|^2_{h^N}\,.
\end{equation}
Thus, by \eqref{condszegodiag},
\begin{eqnarray*} \log \|\Pi_N^p(z,z)\|_{h^N}  & = &  \log \left( \|\Pi_N(z,z)
\|_{h^N} -  \frac{\|\Pi_N(z,p)\|_{h^N}^2}{\|\Pi_N(p,p)\|_{h^N}}\right)\nonumber
\\
& = &  \log \|\Pi_N(z,z)\|_{h^N} +  \log \left( 1 -
P_N(z,p)^2\right)  .  \end{eqnarray*}

By \eqref{conditionalPL} and \eqref{condszegodiag},
\begin{align}\label{conda}
\E_N\big(Z_s\big|s(p)=0\big) &=\frac i{2\pi}\ddbar \log
\|\Pi_N(z,z)\|_{h^N} +\frac i{2\pi} \ddbar \log \left( 1 -
            P_N(z,p)^2\right)+\frac N\pi \om_h\notag\\
&=\E_N\big(Z_s\big)+\frac i{2\pi} \ddbar \log \left( 1 -
            P_N(z,p)^2\right)\,,\end{align}
            concluding the proof of the Proposition.\end{proof}

Theorem \ref{scaled} involves the conditional zero current of a system of
random sections, which we now define precisely:

\begin{defin} \label{defcondk} Let $(L,h)$ be a positive Hermitian holomorphic
line bundle
 over a compact \kahler manifold $M$ with \kahler form $\om_h$, let $1\le k\le
m=\dim M$,
  and let $p\in M$.  Let  $N\gg 0$ and give $H^0(M,L^N)$ the induced Hermitian
Gaussian measure $\ga_N$.
   We let $$T: \bigoplus^kH^0(M,L^N)\to \oplus^kL^N_p\,,$$ where $\bigoplus^kV$
denotes $k$-tuples in $V$.
The conditional expected zero current $\E_N\big(Z_{s_1,\dots,s_k}\big|
s_1(p)=v_1,\,\dots,\,s_k(p)=v_k\big)$  is defined by:
$$\bigg(\E_N\big(Z_{s_1,\dots,s_k}\big|s_1(p)=v_1,\,\dots,\,s_k(p)=v_k\big)\,,\,
\phi\bigg)\ =\ \E_{\ga_N^k}\big((Z_{s_1,\dots,s_k},\phi)\big|T=(v_1,\dots,v_k)
\big)\,,$$ for smooth test forms $\phi\in \dcal^{m-k,m-k}(M)$.
The conditional expected zero distrbution is the current $$K_k^N(z|p):= \E_N\big(Z_{s_1,\dots,s_k}\big|
s_1(p)=0,\,\dots,\,s_k(p)=0\big)\,,$$ which is well defined according to the following lemma.
\end{defin}

\begin{lem}\label{continuity} For $N\gg 0$, the  mapping $$ (v_1,\dots,v_k)
\mapsto\E_N\big(Z_{s_1,\dots,s_k}\big|s_1(p)=v_1,\,\dots,\,s_k(p)=v_k\big)$$ is
 a continuous map  from $\bigoplus^kL^N_p$ to $\dcal'^{m-k,m-k}(M)$.
\end{lem}

\begin{proof} Let
$$E=\{(s_1,\dots,s_k)\in \bigoplus^kH^0(M,L^N):\dim Z_{s_1,\dots,s_k} = n-k\}
\,.$$
Since $L$ is ample, for $N$ sufficiently large, $E\cap T\inv(v_1,\dots,v_k)$ is
a proper algebraic subvariety of $T\inv(v_1,\dots,v_k)$ and hence has Lebesgue
measure 0 in $T\inv(v_1,\dots,v_k)$, for all $(v_1,\dots,v_k)\in \oplus^kL^N_p
$.  Then Proposition \ref{special} applies with $\C^n$ replaced by $
\oplus^kH^0(M,L^N)$, and continuity follows exactly as in the proof of Lemma
\ref{continuity1}.\end{proof}

\section{Proof of Theorem \ref{unscaled}}\label{proof1}

\subsection{Proof for $k=1$}

We first prove Theorem \ref{unscaled} when the condition is that
$s(p) = 0$ for a single point $p$.

\begin{proof}

Let  $\phi\in \dcal'^{m-1,m-1}(M)$ be a smooth test form. By
Proposition \ref{cond}, we have
\begin{equation}\label{cond1} \Big( \E_N(Z_s:s(p)=0),\phi\Big)\ =\
(\E_N Z_s,\phi) +\int_M  \log \left( 1 - P_N(z,p)^2\right) \frac
i{2\pi} \ddbar\phi\,.\end{equation}

Away from the diagonal, we can write $\log \left( 1 -
P_N(z,p)^2\right)=P_N(z,p)^2+\frac12P_N(z,p)^4+\cdots$, and we
have by \eqref{far},
\begin{equation} \label{ddbar1} \log \left( 1 -  P_N(z,p)^2\right) = O(N^{-
m-2})\qquad \mbox{uniformly for }\ d(z,p)\ge
b\,\sqrt{\frac {\log N}{N}}, \end{equation} where $b=\sqrt{2m+6}$.
 Furthermore  by
\eqref{ddbar1}, we have

\begin{align*}&\int_M  \log \left( 1 - P_N(z,p)^2\right) \frac i{2\pi}
\ddbar\phi =\int_{d(z,p)\leq b\sqrt{\frac{\log N}{N}}} \log \left( 1 -
P_N(z,p)^2\right) \frac i{2\pi} \ddbar\phi +O(N^{-m-2})\,.\end{align*}
Using local normal coordinates $(w_1,\dots,w_m)$  centered at $p$, we
 write $$\frac i{2\pi}\ddbar \phi = \psi(w)\,\Om_0(w)\,,\qquad \Om_0(w)=
  \left( \frac i2\right)^m dw_1\wedge d\bar w_1\wedge \cdots \wedge dw_m\wedge
d\bar w_m\,.$$  Recalling \eqref{near}, we then have
\begin{align}\int_M  \log &\left( 1 - P_N(z,p)^2\right) \frac i{2\pi}
\ddbar\phi\notag\\& =\int_{|w|\leq b\sqrt{\frac{\log N}{N}}}\log
\left[1-P_N(p+w,p)^2\right]\psi(w)\,\Om_0(w)+O(N^{-m-2})\notag\\&
=N^{-m}\int_{|u|\leq b\sqrt{\log N}}\log\left[1-P_N\left(p+\frac u{\sqrtn},p
\right)^2\right]\psi\left(\frac u{\sqrtn}\right)\Om(u)+O(N^{-m-2})\,.
\label{intM}
\end{align}

Let
\begin{equation}\label{Lambda} \Lambda_N(z,p)= -\log
P_N(z,p)\;.\end{equation} so that
\begin{equation} \label{USEFUL} \log \left( 1 -
P_N(z,p)^2\right)=Y\circ \Lambda_N(z,p)\;,
\end{equation} where
\begin{equation}\label{Y}Y(\la):=   \log (1-e^{-2\la})\quad \mbox{for }\
\la>0 .\end{equation}
 By \eqref{near}--\eqref{nearr},
\begin{equation}\label{RN0}\Lambda_N\left(p+\frac u{\sqrtn}\,,p\right)
={\half |u|^2} + \wt R_N(u) \;,\end{equation} where
\begin{equation}\label{RN}\wt
R_N(u)=-\log[1+R_N(u,0)]=O(|u|^2N^{-1/2+\ep})\quad \mbox{for }\
|u|<b\sqrt{\log N}\;.\end{equation} We note that \begin{equation}\label{Yest}
0<-Y(\la) = - \log(1-e^{-2\la}) \le \left(1+\log^+
\frac1{\la}\right)\;,\end{equation}
\begin{equation}\label{Y'}Y'(\la)= \frac
2{e^{2\la}-1}\le \frac 1{\la},\quad \mbox{for }\ \la>1\;.\end{equation}
Hence by \eqref{USEFUL}--\eqref{Y'}, \begin{equation}\label{L} \log\left[1-P_N
\left(p+\frac u{\sqrtn},p\right)^2\right] = \log \left(1- e^{-|u|^2}\right) +
O(N^{-1/2+\ep})\qquad \mbox{for }\ |u|<b\sqrt{\log N}\,.\end{equation} Since $
\psi\left(\frac u{\sqrtn}\right) =\psi(0)+O\left(\frac u{\sqrtn}\right)$, we
then have
\begin{multline*} \log\left[1-P_N\left(p+\frac u{\sqrtn},p\right)^2\right]\psi
\left(\frac u{\sqrtn}\right) = \psi(0)\log \left(1- e^{-|u|^2}\right) +
O(N^{-1/2+\ep})\\ +\frac 1\sqrtn\,O\left(|u|\,|\log (1- e^{-|u|^2})|\right)
\qquad \mbox{for }\ |u|<b\sqrt{\log N}\,.\end{multline*}
 Since $O\left((\log N)^mN^{-1/2+\ep}\right)=O(N^{-1/2+2\ep})$ and  $|u|\log
(1- e^{-|u|^2}) \in L^1(\C^m)$, we conclude that
\begin{multline*}\int_{|u|\leq b\sqrt{\log N}}\log\left[1-P_N\left(p+\frac
u{\sqrtn},p\right)^2\right]\psi\left(\frac u{\sqrtn}\right)\Om_0(u) \\= \psi(0)
\int_{|u|\leq b\sqrt{\log N}}\log\left[1-e^{-|u|^2}\right]\Om_0(u) + O(N^{-1/2+
\ep}).\end{multline*}
We note that $$\int_{|u|\geq b\sqrt{\log N}}\log\left[1-e^{-|u|^2}\right]
\Om_0(u) = \frac{2\pi^m}{(m-1)!}\int_{b\sqrt{\log N}}^{+\infty} \log(1-
e^{r^2})r^{2m-1}\,dr = O\left(N^{-b^2/2}\right)\,.$$  Since $b>1$, we then have
\begin{multline}\label{intest}\int_{|u|\leq b\sqrt{\log N}}\log\left[1-P_N
\left(p+\frac u{\sqrtn},p\right)^2\right]\psi\left(\frac u{\sqrtn}\right)
\Om_0(u) \\= \psi(0) \int_{\C^m}\log\left[1-e^{-|u|^2}\right]\Om_0(u) +
O(N^{-1/2+\ep}).\end{multline}

Combining \eqref{cond1}, \eqref{intM} and \eqref{intest}, we have
\begin{equation} \label{cond2} \Big( \E(Z_s:s(p)=0),\phi\Big)\ =\ (\E Z_s,\phi)
+ N^{-m} \psi(0) \int_{\C^m}\log\left[1-e^{-|u|^2}\right]\Om_0(u) + O(N^{-
m-1/2+\ep})\,.\end{equation}
We note that\begin{equation} \label{psi0}
\psi(0)=\frac 1{2\pi}\,\frac {i\ddbar \phi(p)}{\Om_M(p)}\end{equation} and
\begin{eqnarray}\int_{\C^m}\log\left[1-e^{-|u|^2}\right]\Om_0(u)&=&
\frac{2\pi^m}{(m-1)!}\int_0^{+\infty} \log(1-e^{-r^2})r^{2m-1}\,dr \notag\\&=&
\frac{\pi^m}{(m-1)!}\int_0^{+\infty} \log(1-e^{-t})t^{m-1}\,dt\notag \\&=&-
\frac{\pi^m}{(m-1)!}\sum_{n=1}^{+\infty}\int_0^{+\infty} \frac{e^{-nt}}
{n}t^{m-1}\,dt\notag\\
&=& -\frac{\pi^m}{(m-1)!}\sum _{n=1}^{+\infty}\frac {(m-1)!}{n^{m+1}} \ =\ -
\pi^m\, \zeta(m+1)\,.
\label{zeta}\end{eqnarray}  The the one-point case ($k=1$) of Theorem
\ref{unscaled} follows by substituting \eqref{psi0}--\eqref{zeta} into
\eqref{cond2}.

\end{proof}

\subsection{The multi-point case}

We now condition on vanishing at $k$ points $p_1,\dots,p_k$.

\begin{proof}

We   let $H^V_N\subset H^0(M, L^N)$ denote the space of  holomorphic
sections vanishing at the points $p_1,\dots,p_k$. Let $\Phi_N^{p_j}$ be the
coherent state at $p_j$ (given by \eqref{coherent}--\eqref{coherent1}) for
$j=1,\dots,k$.  By \eqref{orthog}, a section $s\in H^0(M,L^N)$ vanishes at $p_j
$ if and only if $s$ is orthogonal to $\Phi_N^{p_j}$.  Thus $H^0(M,L^N)= H^V_N
\oplus Span\{\Phi_N^{p_j}\}$.   Let $$T:H^0(M,L^N)\to L_{p_1}^N\oplus\cdots
\oplus L_{p_k}^N\,,\quad s\mapsto s(p_1)\oplus\cdots s(p_k)\,,$$ so that $\ker
T= H^V_N$.  By Lemma \ref{linear}, the conditional expectation is given by
\begin{equation} \label{conditional2} \E_N\big(Z_s\big|s(p_1)=
\cdots=s(p_k)=0\big) = \frac i{2\pi} \ddbar\log\|\Pi_N^V(z,z)\|_{h^N} +\frac N
\pi \om_h\,,\end{equation} where $\Pi_N^V$ is the conditional \szego kernel for
the projection onto $\Pi_N^V$. We let $\Pi_N^\perp(z,w)$ denote the kernel for
the orthogonal projection onto $(H^V_N)^\perp=Span\{\Phi_N^{p_j}\}$, so that
\begin{equation}\label{Pi-orthog} \Pi_N^V(z,w)=\Pi_N(z,w)-\Pi_N^\perp(z,w)\;.
\end{equation}

Recalling \eqref{coherent}--\eqref{coherent1}, we have
\begin{eqnarray*}\lla \Phi_N^{p_i}, \Phi_N^{p_j}\rra\,
\overline{e_{p_i}^{\otimes N}} \otimes e_{p_j}^{\otimes N} &=&
\frac{ \lla \sum_\al S^N_\al(z)\otimes \overline
{S^N_\al(p_i)}\,,\,\sum_\beta  S^N_\be(z)\otimes \overline
{S^N_\be(p_j)}\rra}
{\|\Pi_N(p_i,p_i)\|_{h_N}^{1/2}\|\Pi_N(p_j,p_j)\|_{h_N}^{1/2}}\\[6pt]& = &
\frac{\overline{\Pi_N(p_i,p_j)}}
{\|\Pi_N(p_i,p_i)\|_{h_N}^{1/2}\|\Pi_N(p_j,p_j)\|_{h_N}^{1/2}}
\ ,\end{eqnarray*} and therefore by \eqref{far},
\begin{equation}\label{inner} \left|\lla \Phi_N^{p_i}, \Phi_N^{p_j}\rra\right|
= P_N(p_i,p_j)= \de_i^j+O(N^{-\infty})\,.\end{equation} In particular the $
\Phi_N^j$  are linearly independent, for $N\gg 0$. Let $$\lla \Phi_N^{p_i},
\Phi_N^{p_j}\rra =\de_i^j +W_{ij}\,.$$ By \eqref{inner}, $W_{ij}=O(N^{-\infty})
$.  Let us now replace the basis $\{\Phi_N^{p_j}\}$ of $(H_N^V)^\perp$ by an
orthonormal basis
$\{\Psi_N^{j}\}$, and write $$ \Psi_N^{i}=\sum_{j=1}^k A_{ij}\,\Phi_N^{p_j}\,.
$$
Then $$\de_i^j =\lla \Psi^i_N,\Psi^j_N \rra = \sum_{\al,\be} \lla A_{i\al}
\Phi^{p_\al},
A_{j\be} \Phi^{p_\be}\rra = \sum_{\al,\be} A_{i\al}\overline A_{j\be}(\de_\al^
\be+W_{\al\be})\,,$$
or $I=A(I+W)A^*$.

We have $$\Pi_N^\perp(z,z)= \sum \Psi_N^j(z)\otimes \overline{ \Psi_N^j(z)}=
\sum_{j,\al,\be}A_{j\al} \overline A_{j\be}\Phi_N^{p_\al}\otimes \overline
{\Phi_N^{p_\be}}=
\sum_{j\be}B_{\al\be}\Phi_N^{p_\al}\otimes \overline {\Phi_N^{p_\be}}\,,$$
where \begin{equation}\label{B}B={}^t\!A\,\overline A= {}^t(A^*\,A)= {}^t(I+W)
\inv= I+O(N^{-\infty})\,.\end{equation}  The final equality in \eqref{B}
follows by noting that
$$\|W\|_{HS}=\eta<1\implies \|(I+W)\inv-I\|_{HS}=\|W-W^2+W^3+\cdots\|_{HS} \le
\eta+\eta^2+\eta^3+\cdots = \frac \eta{1-\eta},$$  where $\|W\|
_{HS}=[\mbox{Trace}(WW^*)]^{1/2}$ denotes the Hilbert-Schmidt norm.
 Therefore $$\|\Pi_N^\perp(z,z)\| = \sum_{j=1}^k\|\Phi_N^{p_j}(z)\|^2+O(N^{-
\infty})\,.$$
Repeating the argument of the 1-point case, we then obtain
\begin{align}\label{condk}
\E_N\big(Z_s\big|s(p_1)=\cdots=s(p_k)=0\big) =\ (\E_N Z_s,\phi) +  \log \left( 1 -
\sum P_N(z,p_j)^2\right)  +O(N^{-\infty}).\end{align}
It suffices to verify the  theorem in a neighborhood of an arbitrary point
$z_0\in M$.  If $z_0\not\in \{p_1,\dots,p_k\}$, then $\log \left( 1 - \sum
P_N(z,p_j)^2\right)=O(N^{-\infty})$ in a neighborhood of $z_0$, and the formula
trivially holds. Now suppose $z_0=p_1$, for example.  Then
$$\log \left( 1 - \sum P_N(z,p_j)^2\right)=\log \left( 1 -  P_N(z,p_1)^2\right)
+O(N^{-\infty})$$ near $p_1$ and the conclusion holds there by the computation
in the 1-point case.\end{proof}

\section{Proof of Theorem \ref{scaled}: The scaled conditional expectation}
\label{proof2}

In this section we shall prove  Theorem \ref{scaled} together with the
following analogous result on the scaling asymptotics of conditional expected
zero currents of dimension $\ge 1$:

\begin{theo}\label{all codim} Let $1\le k\le m-1$. Let $(L,h)\to (M,\om_h)$ and
$(H^0(M,L^N),\gamma_{h}^N)$ be as in Theorem \ref{unscaled}. Let $p\in M$, and
choose normal coordinates $z=(z_1,\dots,z_m):M_0,p\to \C^m,0$ on a neighborhood
$M_0$ of $p$.
Let $\tau_N=\sqrtn\,z:M_0\to\C^m$ be the scaled coordinate map.  Then for a
smooth test form
$\phi\in\dcal^{m-k,m-k}(\C^m)$, we have
$$\Big( K^N_k(z|p),\tau_N^*\phi\Big)=  \int_{\C^m\sm\{0\}} \phi\wedge
\left(
\frac i{2\pi}\ddbar\left[\log(1-e^{-|u|^2})
+|u|^2\right]\right)^k \ +\ O(N^{-1/2+\epsilon})\,,
$$ and thus
$$\tau_{N*}\Big( K^N_k(z|p)\Big)
\to K_{km}^\infty(u|0):=\left(
\frac i{2\pi}\ddbar\left[\log(1-e^{-|u|^2}) +|u|^2\right]\right)^k\,,$$ where
$u=(u_1,\dots,u_m)$ denotes the coordinates in $\C^m$.
\end{theo}

Just as in Theorem \ref{scaled}, 
$K_{km}^\infty(u|0)$ is the conditional expected zero current
of $k$ independent random  functions in the  Bargmann-Fock ensemble on $\C^m$.  

To prove Theorems \ref{scaled} and \ref{all codim}, we first note that by
\eqref{Zel} and Proposition \ref{cond}, we have
\begin{align}\label{cond3}
 K^N_1(z|p)
&=\frac i{2\pi}\ddbar \log \|\Pi_N(z,z)\|_{h^N} +\frac i{2\pi} \ddbar \log
\left( 1 -
            P_N(z,p)^2\right)+\frac N\pi \om_h\notag \\
&=\frac N\pi \om_h+\frac i{2\pi} \ddbar \log \left( 1 -
P_N(z,p)^2\right)+O(N \inv)\,.\end{align}
In normal coordinates $(z_1,\dots,z_m)$  about $p $, we have
\begin{equation}\label{om}\om_h=\frac i2 \sum g_{jl}dz_j\wedge d
\bar z_l\,,\quad g_{jl}(z)=\de_j^l +O(|z|).\end{equation} Changing
variables to  $u_j=\sqrtn z_j$ gives
\begin{equation}\label{omega}\frac N\pi \om_h =\frac i{2\pi}\sum g_{jl}
\left(\frac u\sqrtn\right)\, du_j\wedge d\bar u_l = \frac i{2\pi}
\ddbar|u|^2 + \sum O(|u|N^{-1/2})\, du_j\wedge d\bar
u_l\,.\end{equation}

We can now easily verify the one  dimensional case of Theorem \ref{scaled}: Let
$m=1$.  By \eqref{L}, \eqref{cond3} and \eqref{omega}, we have
$$\Big( K^N_1(z|p),\tau_N^* \phi\Big)\ =\ \frac i{2\pi}\int_\C \left[\log(1-
e^{-|u|^2}) +|u|^2\right] \ddbar\phi +O(N^{-1/2+\ep})\,$$ for a
smooth test function $\phi\in\dcal(\C)$.  By
Green's formula,
\begin{eqnarray*}\int_{|u|>\ep} \left[\log(1-e^{-|u|^2}) +|u|^2\right]\ddbar
\phi &=& \int_{|u|>\ep}\phi\,\ddbar\left[\log(1-e^{-|u|^2}) +|u|^2\right]\\&& -
\frac{i\ep^2}{1-e^{-\ep^2}}\int_{|u|=\ep}\phi\,d\theta+O(\ep\log\ep)\\ &\to &
\int_\C\phi\,\ddbar\left[\log(1-e^{-|u|^2}) +|u|^2\right]-2\pi i\,\phi(0)\,,
\end{eqnarray*} which yields Theorem \ref{scaled} for $k=m=1$.

For the  dimension $m>1$ cases, we first derive some pointwise formulas on $M
\sm\{p\}$:
 Let $\La_N(z)=\La_N(z,p) = -\log P_N(z,p)$. Recalling \eqref{USEFUL}, we have

\begin{eqnarray*}\ddbar\log \left( 1 - P_N(z,p)^2\right) &=& \ddbar (Y\circ
\La_N)\ = \ Y''(\La_N)\,\d\La_N\wedge\dbar \La_N +Y'(\La_N)\,\ddbar\La_N\\
&=& -\frac{4\,e^{-2\La_N}}{(1-e^{-2\La_N})^2}\,\d\La_N\wedge\dbar \La_N +\frac
2{e^{2\La_N}-1}\,\ddbar\La_N
\,.\end{eqnarray*}

 By \eqref{nearr} and \eqref{RN0}--\eqref{RN}, we have
$$\La_N= \half |u|^2 + O(|u|^2N^{-1/2+\ep})\,,\quad \frac {\d\La_N}{\d\bar
u_j}= \half  u_j + O(|u|N^{-1/2+\ep})
\,,\quad  \frac {\d^2\La_N}{\d u_j\d\bar u_l}= \half \de_j^l+ O(N^{-1/2+\ep})
\,.$$ Thus
$$\dbar\La_N= \half \sum \big[u_j +O(|u|N^{-1/2+\ep})\big]\,d\bar u_j\,,$$ and
$$\ddbar\La_N = \left(\half \ddbar |u|^2 + \sum c_{jl} du_j\wedge d\bar u_l
\right),\quad c_{jl}=O(N^{-1/2+\ep}).$$  Since $Y^{(j)}(\la)=O(\la^{-j})$ for
$0<\la<1$, we then have
\begin{eqnarray} \dbar\log \left( 1 - P_N(z,p)^2\right)&=& \big[Y'(\half |u|^2)
+ O(|u|^{-2}N^{-1/2+\ep})\big]\big[\half\dbar  |u|^2 +\sum O(|u|N^{-1/2+\ep})d
\bar u_j\big]\notag\\
&=& \frac 1{e^{|u|^2}-1}\,\dbar |u|^2 +\sum O(|u|^{-1}N^{-1/2+\ep})\, d\bar u_j
\,,\label{dbarlog}
\\ \ddbar\log \left( 1 - P_N(z,p)^2\right)&=&-\frac{e^{-|u|^2}}{(1-e^{-|u|
^2})^2}\,\d|u|^2\wedge \dbar |u|^2+\frac 1{e^{|u|^2}-1}\,\ddbar |u|^2\notag\\&&
+\sum O(|u|^{-2}N^{-1/2+\ep})\,du_j\wedge d\bar u_k\notag\\&=&\ddbar\log(1-
e^{-|u|^2})+\sum O(|u|^{-2}N^{-1/2+\ep})\,du_j\wedge d\bar u_l\,,
\label{ddbarlog}\end{eqnarray} for $0<|u|<b$.
Therefore by \eqref{cond3}, \eqref{ddbarlog} and \eqref{omega},
\begin{eqnarray} K^N_1(z|p) &=&\frac i{2\pi}\, \frac 1{1-e^{-|
u|^2}}\left[
-\frac{e^{-|u|^2}}{1-e^{-|u|^2}}\,\d|u|^2\wedge\dbar |u|^2 +\ddbar|u|^2\right]
\notag\\&& +\sum
O(|u|^{-2}N^{-1/2+\ep})\,du_j\wedge d\bar u_k\notag\\ &=&\frac i{2\pi}\ddbar
\left[\log(1-e^{-|u|^2})
+|u|^2\right]+\sum O(|u|^{-2}N^{-1/2+\ep})\,du_j\wedge d\bar u_l,\quad
\label{EZs}\end{eqnarray}
for $0<|u|<b$.

 We shall use the following notation: If $R\in\dcal'^r(M)$ is a current of
order 0 (i.e., its coefficients are given locally by measures), we
  write $R=R_{sing}+R_{ac }$, where $R_{sing}$ is supported on a set of
(volume) measure 0, and  the coefficients of  $R_{ac }$ are in
$L^1_{loc}$. We also let $\|R\|$ denote the total variation
measure of $R$:
$$(\|R\|,\psi):= \sup \{|(R,\eta)|: \eta\in \dcal^{2m-r}(M), |\eta|\le \psi\},
\quad \mbox{for }\ \psi\in\dcal(M)\,.$$

\begin{lem}\label{product} The conditional expected zero distributions are  given by
\begin{eqnarray*}K^N_k(z|p)& =& \left[\frac i{2\pi} \ddbar\log\|\Pi_N(z,z)\|_{h^N}
+ \frac i{2\pi} \ddbar\log  \left( 1 - P_N(z,p)^2\right)+\frac N\pi \om_h
\right]^k_{ac}\\&& \mbox{for }\ 1\le k\le m-1\,,\\[8pt]K^N_m(z|p)& =& \de_p+\left[\frac i{2\pi} \ddbar\log\|\Pi_N(z,z)\|_{h^N}
+ \frac i{2\pi} \ddbar\log  \left( 1 - P_N(z,p)^2\right)+\frac N\pi \om_h
\right]^m_{ac}\,.\end{eqnarray*}
In particular, the currents $K^N_k(z|p)$ are smooth forms on $M\sm\{p\}$ for $1\le k\le m$, and only the top-degree current $K^N_m(z|p)$ has point  mass at $p$.
\end{lem}

\begin{proof} Let $$T:H^0(M,L^N)^k\to (L^{\otimes N}_p)^k\,,\quad
(s_1,\dots,s_k)\mapsto(s_1(p),\dots,s_k(p))\,.$$
By Proposition \ref{special} and Definition \ref{defcondk},
$$\left(K^N_k(z|p),\phi\right)= \E_{(\ga^p_N)^k}(Z_{s_1,\dots,s_k},\phi)\,,$$ for $\phi\in
\dcal^{m-k,m-k}(M\sm \{p\})$, where $\ga^p_N$ is the conditional Gaussian on $H^p_N$.

Next, we shall apply Proposition 2.2 in \cite{SZa} to show that \begin{equation}\label{E^k} K^N_k(z|p)=
\E_{(\ga^p_N)^k}
(Z_{s_1,\dots,s_k}) = \left[\E_{\ga^p_N}Z_{s}\right]^{\wedge
k}=\big[K^N_1(z|p)\big]^{\wedge k} \quad
\mbox{on }\ M\sm \{p\}\,.\end{equation}  We cannot apply Proposition 2.2 in \cite{SZa} directly, since all sections of $H^p_N$ vanish at $p$ by definition, so $H^p_N$ is not base point free. Instead, we shall apply this result to the blowup $\wt M$ of $p$.
Let $\pi:\wt M\to M$ be the blowup map, and let $E=\pi\inv(p)$
denote the exceptional divisor. Let $\wt L\to \wt M$ denote the pullback  of $L$, and let  $\ocal(-E)$ denote the line bundle
over $\wt M$ whose local sections are holomorphic functions
vanishing on  $E$ (see \cite[pp.~136--137]{GH}).  Thus  we have  isomorphisms \begin{equation}\label{pistar}\tau_N:H_N^p\buildrel{\approx}\over\to H^0(\wt M,\wt L^N\otimes \ocal(-E))\,,\qquad \tau_N(s)=s\circ \pi\,.\end{equation} (Surjectivity follows from Hartogs' extension theorem; see, e.g., \cite[p.~7]{GH}.) 

Let $\ical_p\subset \ocal_M$ denote the maximal ideal sheaf of $\{p\}$. From the long exact cohomology sequence
$$\cdots\to H^0(M,\ocal(L^N) )\to H^0(M,\ocal(L^N)\otimes (\ocal_M/\ical^2_p))\to
H^1(M,\ocal(L^N)\otimes\ical^2_p)\to\cdots$$ and the Kodaira vanishing theorem, it follows that $H^1(M,\ocal(L^N)\otimes\ical^2_p)=0$ and thus there exist
sections of $L^N$ with arbitrary 1-jet at $p$,  for $N$ sufficiently large (see, e.g., \cite[Theorem~(5.1)]{SS}). Therefore $\wt L^N\otimes\ocal(-E)$ is base point free.

We give $H^0(\wt M,\wt L^N\otimes\ocal(-E))$ the Gaussian measure $\wt\ga_N:=\tau_{N*}\ga^p_N$. By \cite[Prop.~2.1--2.2]{SZa} applied to  the line
bundle $\wt L^N\otimes\ocal(-E)\to \wt M$ and the space
$\scal=H^0(\wt M,\wt L^N\otimes\ocal(-E))$, we have $\E_{(\wt\ga_N)^k}\left(Z_{\tilde
s_1,\dots,\tilde s_k}\right)= \left(\E_{\wt\ga_N} Z_{\tilde
s_1}\right)^{\wedge k}$ (where the $\tilde s_j$ are independent random sections in $\scal$). 
Equation \eqref{E^k} then follows by identifying $\wt M\sm E$ with $M\sm\{p\}$ and
$H^0(\wt M,\wt L^N\otimes\ocal(-E))$ with $H_N^p$. 
By equations \eqref{cond3} and \eqref{E^k}, we then have
\begin{eqnarray*}K^N_k(z|p)& =& \left[\frac i{2\pi} \ddbar\log\|\Pi_N(z,z)\|_{h^N}
+ \frac i{2\pi} \ddbar\log  \left( 1 - P_N(z,p)^2\right)+\frac N\pi \om_h
\right]^k\\&& \mbox {on }\ M\sm\{p\}\,, \quad  \mbox{for }\ 1\le k\le m\,.
\end{eqnarray*}

Since $K^N_k(z|p)$ is a current of order 0, to complete the proof of the lemma it suffices to show that 

\begin{enumerate} \item[i)] \ $\|K^N_k(z|p)\|(\{p\}) = 0$ for $k<m$,
\item[ii)] \ $K^N_m(z|p)(\{p\}) =1$.\end{enumerate}  

We first verify (ii): Let $\{\phi_n\}$ be a decreasing sequence of smooth functions on $M$ such that $0\le\phi_n\le 1$ and $\phi_n\to \chi_{\{p\}}$ as $n\to\infty$.  We consider the random variables $X^m_n:(H_p^N)^m\to \R$ given by $$X^m_n(\sss)=(Z_\sss,\phi_n)\,,\quad \sss=(s_1,\dots,s_m)\,.$$
Every $m$-tuple $\sss\in (H^p_N)^m$
has a  zero at $p$ by definition, and almost all $\sss$ have only simple zeros; therefore $X^m_n(\sss)\to Z_\sss(\{p\}) = 1$ a.s. Furthermore $1\le X^m_n(\sss)\le (Z_\sss,1) = N^mc_1(L)^m$.  Therefore by dominated convergence,
$$K^N_m(z|p)(\{p\}) =\lim_{n\to\infty} (K^N_m(z|p),\phi_n) = \lim_{n\to\infty} \int X^m_n\,d(\ga^p_N)^m= \int \lim_{n\to\infty} X^m_n\,d(\ga^p_N)^m=1\,.
$$  To verify (i), we note that $\|K^N_k(z|p)\|\,\Om_M\le C  K^N_k(z|p)\wedge \om_h^{m-k}$ (where the constant $C$ depends only on $k$ and $m$), and thus it suffices to show that \begin{enumerate}\item[i$'$)] \  $ \big(K^N_k(z|p)\wedge \om_h^{m-k}\big)
(\{p\})=0$ for $k<m$. \end{enumerate}
  For $k<m$, we let $$X^k_n(\sss)=(Z_\sss\wedge \om_h^{m-k},\phi_n)\le \pi^{m-k}N^mc_1(L)^m\,,\quad \sss=(s_1,\dots,s_k)\,,$$ where $\phi_n$ is as before.
But this time, $X^k_n(\sss)=\int_{Z_\sss}\phi_n \om_h^{m-k}\to 0$ a.s.
Equation (i$'$) now follows exactly as before. (Equation (i) is also an immediate consequence of Federer's support theorem for locally flat currents 
\cite[4.1.20]{Fe}.)

\end{proof}

We now complete the proof of Theorem \ref{all codim}: By Lemma \ref{product} and the asymptotic formula  \eqref{EZs},  we have
\begin{eqnarray*}K^N_k(z|p)& =& K^N_k(z|p)_{ac }\\&=&\left[\frac i{2\pi} \ddbar
\log\|\Pi_N(z,z)\|_{h^N}
+ \frac i{2\pi} \ddbar\log  \left( 1 - P_N(z,p)^2\right)+\frac N\pi \om_h
\right]^k_{ac }\\&=& \left(\frac i{2\pi}\right)^k
\left[-\frac{e^{-|u|^2}}{(1-e^{-|u|^2})^2}\,\d|u|^2\wedge\dbar |u|^2 +
\frac{\ddbar|u|^2}{(1-e^{-|u|^2})} \right]^{k} \\&&\quad +\ \sum O(
|u|^{-2k}N^{-1/2+\ep})du_{j_1}\wedge du_{l_1}\wedge\cdots \wedge du_{j_k}\wedge
du_{l_k}\,.\end{eqnarray*}

Therefore,
\begin{eqnarray}\left(K^N_k(z|p)\,,\,\tau_N^*\phi\right)&=&\int_{M_0\sm\{p\}}
\left[\frac i{2\pi} \ddbar\log\|\Pi_N(z,z)\|_{h^N}
+ \frac i{2\pi} \ddbar\log  \left( 1 - P_N(z,p)^2\right)+\frac N\pi \om_h
\right]^k\wedge \tau_N^*\phi\notag\\&=& \left(\frac i{2\pi}\right)^k
\int_{\C^m\sm\{0\}}\left[-\frac{e^{-|u|^2}}{(1-e^{-|u|^2})^2}\,\d|u|^2\wedge
\dbar |u|^2 +\frac{\ddbar|u|^2}{(1-e^{-|u|^2})} \right]^{k}\wedge\phi\notag
\\&&\quad +\ N^{-1/2+\ep}\|\phi\|_\infty \int_{\supp(\phi)}O(|u|^{-2k})\,(i
\ddbar|u|^2)^m\,,\label{best}\end{eqnarray} which verifies
Theorem \ref{all codim}.\qed

\medskip
To prove Theorem \ref{scaled}, we need to integrate by parts, since if $k=m$, the integral in the last line of \eqref{best} does not a priori converge.   To begin the proof, by Lemma \ref{product} we have \begin{multline} 
\big(K^N_m(z|p)\,,\,\phi\circ\tau_N\big)=\phi(0)\\+
\int_{M_0\sm\{p\}}\phi\big(\sqrt N\,z\big)\left[\frac i{2\pi} \ddbar\log\|
\Pi_N(z,z)\|_{h^N}
+ \frac i{2\pi} \ddbar\log  \left( 1 - P_N(z,p)^2\right)+\frac N\pi \om_h
\right]^m.\label{withdelta}\end{multline}
Writing \begin{equation}\label{rho}\om_h=\frac i2 \ddbar \rho\,,\qquad \rho(z)
= |z|^2+O(|z|^3)\,,\end{equation} we then have
\begin{equation}\label{withdelta1} \big(K^N_m(z|p)\,,\,\phi\circ\tau_N\big)=
\phi(0)+ \int_{\C^m\sm\{0\}}\phi \,\cdot \left(\frac i{2\pi}\ddbar f_N\right)^m
\,,\end{equation} where
$$f_N(u)=\log \left\|\Pi_N\left(\frac u \sqrtn,\frac u \sqrtn\right)\right\|
_{h^N} -m\log (N/\pi)
+\log  \left( 1 - P_N\left(\frac u\sqrtn,0\right)^2\right)+N\rho\left(\frac u
\sqrtn\right).$$  By \eqref{Zel}, \eqref{L} and \eqref{rho},
\begin{equation}\label{fN}f_N(u)=  \log \left(1- e^{-|u|
^2}\right)  +|u|^2+ O(N^{-1/2+\ep})\,.\end{equation} Again recalling
\eqref{EZs}, we have
\begin{equation}\label{ddfN}\ddbar f_N= -\frac{e^{-|u|^2}}{(1-e^{-|u|^2})^2}\,
\d|u|^2\wedge\dbar |u|^2 +\frac{\ddbar|u|^2}{(1-e^{-|u|^2})} + O(
|u|^{-2}N^{-1/2+\ep})\,.\end{equation}

We now integrate \eqref{withdelta1} by parts.
Let \begin{equation}\label{alpha}\al_N= f_N \,(\ddbar f_N)^{m-1}\,.
\end{equation} Then for $\de>0$,
\begin{equation} \int_{|u|>\de}\phi\,\ddbar\al_N =\int_{|u|>\de} \al_N\wedge
\ddbar\phi +\frac i2\int_{|u|=\de}( \phi\,d^c\al_N -\al_N\wedge d^c\phi)\,,
\end{equation} where $d^c=i(\dbar-\d)$. By \eqref{fN}--\eqref{alpha},
\begin{equation}\label{allalpha} \al_N=\al_\infty + O\left(|u|^{-2m+2}\log(|u|
+|u|\inv)N^{-1/2+\ep}\right)\,,\end{equation} where
\begin{eqnarray*}\al_\infty &=&\left[\log(1-e^{-|u|^2}) +|u|^2\right] \left\{
\ddbar \left[\log(1-e^{-|u|^2}) +|u|^2\right]\right\}^{m-1}
\\&=&
\left[\log(1-e^{-|u|^2}) +|u|^2\right]
\left[-\frac{e^{-|u|^2}}{(1-e^{-|u|^2})^2}\,\d|u|^2\wedge\dbar |u|^2 +
\frac{\ddbar|u|^2}{(1-e^{-|u|^2})} \right]^{m-1}.\end{eqnarray*}
In particular,
\begin{equation}\label{|alpha|}\al_N=O\left(|u|^{-2m+2}\log(|u|+|u|\inv)\right)
\,,\end{equation}  and therefore
$$\lim_{\de\to 0} \int_{|u|=\de}\al_N\wedge d^c\phi = 0\,.$$

Futhermore, by \eqref{dbarlog} and \eqref{ddfN},
$$d^c\al_N = \frac {d^c|u|^2 \wedge (\ddbar |u|^2)^{m-1}}{(1-e^{-|u|^2})^m} +O
\left(|u|^{-2m+1}N^{-1/2+\ep}\right).$$
Therefore,
\begin{eqnarray*}\left(\frac i{2\pi}\right)^m \frac i2
\int_{|u|=\de}
\phi\,d^c\al_N &=& -\frac {\de\cdot \de^{2m-1}}{(1-e^{-
\de^2})^m}\mbox{Average}_{|u|=\de}(\phi) +O(N^{-1/2+\ep})\sup_{|u|=\de}|\phi|
\\&\to & -\phi(0)\,\left[1+O(N^{-1/2+\ep})\right]\,.
\end{eqnarray*}

Thus, \begin{equation} \label{byparts} \left(\frac i{2\pi}\right)^m
\int_{\C^m\sm\{0\}}\phi\,\ddbar\al_N = \left(\frac i{2\pi}\right)^m\int_{\C^m
\sm\{0\}}\al_N\,\ddbar\phi
-\phi(0)\,\left[1+O(N^{-1/2+\ep})\right]\,.\end{equation}
({\it Remark:\/}  In fact, it follows from Demailly's comparison theorem for generalized Lelong numbers \cite[Theorem~7.1]{Dem},
applied to the plurisubharmonic functions $f_N(u)$ and $
\log|u|^2$ and closed positive current $T=1$, that the two measures $i^m \ddbar\alpha_N$ and $i^m \ddbar\log|u|^2$ impart the same mass to the point $0$, and therefore we have the precise identity
$$\left(\frac i{2\pi}\right)^m
\int_{\C^m\sm\{0\}}\phi\,\ddbar\al_N = \left(\frac i{2\pi}\right)^m\int_{\C^m
\sm\{0\}}\al_N\,\ddbar\phi
-\phi(0)\,.$$
However, \eqref{byparts} suffices for our purposes.)

Combining \eqref{withdelta1}, \eqref{allalpha} and \eqref{byparts},
\begin{eqnarray*} \big(K^N_m(z|p)\,,\,\phi\circ\tau_N\big) &=&
\left(\frac i{2\pi}\right)^m \int_{\C^m\sm\{0\}}\al_N\,\ddbar\phi +O(N^{-1/2+
\ep})\\ &=&\left(\frac i{2\pi}\right)^m \int_{\C^m\sm\{0\}}\al_\infty\,\ddbar
\phi +O(N^{-1/2+\ep}).
\end{eqnarray*}

Repeating the integration by parts argument using $\al_\infty$ (or by the above comparison theorem of Demailly \cite{Dem}), we conclude
that $$\left(\frac i{2\pi}\right)^m
\int_{\C^m\sm\{0\}}\phi\,\ddbar\al_\infty = \left(\frac i{2\pi}\right)^m
\int_{\C^m\sm\{0\}}\al_\infty\,\ddbar\phi
-\phi(0)\,.$$ Therefore
\begin{eqnarray*}\big(K^N_m(z|p)\,,\,\phi\circ\tau_N\big) &=&\phi(0) +
\left(\frac i{2\pi}\right)^m \int_{\C^m\sm\{0\}}\phi\ddbar\al_\infty
+O(N^{-1/2+\ep})\\&=&
\phi(0) +\left(\frac i{2\pi}\right)^m \int_{\C^m\sm\{0\}}\phi(u)\left\{
\ddbar \left[\log(1-e^{-|u|^2}) +|u|^2\right]\right\}^m\\&& \qquad
+O(N^{-1/2+\ep})\,,\end{eqnarray*}
which completes the proof of Theorem \ref{scaled}.\qed

\section{Comparison of pair correlation density and conditional
density}\label{comparison} We conclude  with further  discussion
of the comparison  between the pair correlation function  and
conditional Gaussian density of zeros.

\subsection{Comparison  in dimension one}\label{PT}

 We now explain the sense in which the pair correlation
$K^N_{1m}(p)^{-2} K_{2m}^N(z, p)$ of \cite{BSZ,BSZ2} may be viewed as a conditional
probability density.

We begin with the case of  polynomials, i.e. $M=\CP^1$.   The possible zero
sets of a random polynomial form the configuration space
 \begin{equation*}
(\CP^1)^{(N)} = Sym^{N} \CP^1 :=
\underbrace{\CP^1\times\cdots\times \CP^1}_N /S_{N}
\end{equation*}
 of $N$ points of  $\CP^1$, where  $S_N$ is the symmetric
  group on $N$ letters.
We define the  joint probability current of zeros as the
pushforward
  \begin{equation} \label{JPCDEF}\vec K_N^N(\zeta_1, \dots, \zeta_N) : =
\dcal_* \gamma_h^N \end{equation}
  of the Gaussian measure on the space $\pcal_N$ of polynomials of degree $N$
under the  `zero set'
  map $
\dcal: \pcal_N \to (\CP^1)^{(N)}$ taking $s_N $ to its zero set.
An explicit formula for it in local coordinates is

\begin{eqnarray}
    \label{eq-030209b}
    \vec K_N^N(\zeta_1, \dots, \zeta_N) & = & \frac{1}{Z_N(h)}
\frac{|\Delta(\zeta_1, \dots, \zeta_N)|^2  d_2 \zeta_1 \cdots d_2
\zeta_N}{\left(\int_{\CP^1} \prod_{j = 1}^N |(z - \zeta_j)|^2 e^{-
N \phi(z)}  d\nu(z) \right)^{N+1}},
\end{eqnarray}
where $Z_N(h)$ is a   normalizing constant.  We refer to \cite{ZZ}
for further details.

As in \cite[\S 5.4, (5.39)]{D}, the pair correlation
function is obtained from the joint probability distribution by
integrating out all but two variables. If we   fix the second
variable of $K_{2 1}^N(z, p)$ at $p$ and divide by the density
$K_{11}^N(p)$ of zeros at $p$, we obtain the same density as if we
fixed the first variable $\zeta_1 = p$  of the density of $\vec
K_N^N(\zeta_1, \dots, \zeta_N)$, integrated out the last $N - 2$
variables and divided by the density at $p$. But fixing $\zeta_1 =
p$ and  dividing by  $K_{11}^N(p) d_2 \zeta_1$ is the conditional
probability distribution of zeros defined by the random variable
$\zeta_1$.  Thus in
dimension one, $K^N_{11}(p)^{-2} K_{21}^N(z, p)$ is the
conditional density of zeros at $z$ given a zero at $p$ if we
condition using $\zeta_1 = p$ in the configuration space picture.
This use of the term `conditional expectation of zeros given a
zero at $p$' can be found, e.g. in \cite{So}.

\subsection{Comparison in higher dimensions}

The above  configuration space approach is difficult to generalize to higher
dimensions
and full systems of polynomials. In particular, it is difficult
even to describe the configuration of joint zeros of a system as a
subset of the symmetric product. Indeed, the number of  simultaneous zeros
of $m$ sections is almost surely $c_1(L)^mN^m$ so the variety  $C_N$ of
configurations of simultaneous zeros is a subvariety of the symmetric
product $M^{(c_1(L)^mN^m)}$.  Since $C_N$ is the image of the zero set map
$$\dcal:G(m,H^0(N,L^N))\to M^{(c_1(L)^mN^m)}$$from the Grassmannian of 
$m$-dimensional subspaces of $H^0(N,L^N)$, its dimension (given by the
Riemann-Roch formula) is quite small compared with the dimension
of the symmetric product:$$ \dim
C_N=\frac{c_1(L)^m}{(m-1)!}N^m+O(N^{m-1})\sim \frac 1{m!}\dim
M^{(c_1(L)^mN^m)}\,.$$  Under the zero set map, the probability
measure on systems pushes forward to  $C_N$, but to our knowledge there is no explicit
formula as (\ref{eq-030209b}).

We now provide an intuitive and informal comparison of the two scaling limits without using
our explicit formulas.  Let $B_{\delta}(p) \subset \C^m$
be the ball of radius $\delta$ around $p$,
let $\sss = (s_1, \dots, s_m)$ be an  $m$-tuple of independent random
sections in $H^0(M,L^N)$, and let Prob denote the probability measure $(\ga_h^N)^m$ on the space of $m$-tuples $\sss$. We define the events,
 $$U_{\delta}^p = \{\sss: \sss \; \mbox{has a zero in } \;
B_{\delta}(p)\}, \;\;\; U_{\epsilon}^q = \{\sss: \sss \; \mbox{has
a zero in } \; B_{\epsilon}(q)\}.$$

Now the probability interpretation of the pair correlation
function is based on the fact that, as $\delta, \epsilon \to 0$,
$$  \int_{ B_{\delta}(P)
\times B_{\epsilon}(q)} \E \big[Z_{\sss}(z) Z_{\sss}(w)\big] = \mbox{Prob}(U_{\delta}^p \cap U_{\epsilon}^q) \big[1+o(1)\big] \;,
$$ since the probability of having two or more zeros in a small ball is small compared with the probability of having one zero.

It follows that
$$\lim_{\epsilon, \delta
\to 0} \frac{1}{\vol ( B_{\delta}(p)) \times \vol (B_{\epsilon}(q))}
 \mbox{Prob}(U_{\delta}^p \cap U_{\epsilon}^q)  = K^\infty _{2mm}(p, q).
 $$  
Similarly,  $$\lim_{\delta \to 0} \mbox{Prob} ( U_{\delta}^p)
\simeq \frac{1}{\vol B_{\delta}(p)} \int_{ B_{\delta}(p)} \E
Z_{\sss}(z) = K^\infty _{1mm}(p). \;
$$
Hence, as $\epsilon, \delta \to 0$,
$$\mbox{Prob}(U_{\epsilon}^q |
U_{\delta}^p)   \simeq \frac{\left(\int_{ B_{\delta}(p) \times
B_{\epsilon}(q)} \E Z_{\sss}(z) Z_{\sss}(w)\right)}{ \left(\int_{
B_{\delta}(p)} \E Z_{\sss}(z)  \right)} =  \frac{\left(\int_{
B_{\delta}(p) \times B_{\epsilon}(q)} K^\infty _{2mm}(z,w)\right)}{
\left(\int_{ B_{\delta}(p)}K^\infty _{1mm}(z) \right)}, $$ so that
$$\lim_{\epsilon, \delta \to 0}\frac{1}{\vol  B_{\epsilon}(q)}  \mbox{Prob}
(U_{\epsilon}^q |
U_{\delta}^p)   = \frac{K^\infty _{2 mm}(p,q)}{K^\infty _{1mm}(p)}.
$$

By comparison, $$K_1^\infty (q|p) = \lim_{\epsilon \to 0}\frac{1}{\vol
B_{\epsilon}(q)} \mbox{Prob}(U_{\epsilon}^q | \; \sss (p) = 0) =
\lim_{\epsilon, \delta \to 0} \frac{1}{\vol  B_{\epsilon}(q)}
\frac{\mbox{Prob}(U_{\epsilon}^q  \cap \fcal_{\delta}^p )}{\mbox{Prob}
(\fcal_{\delta}^p)},
$$ where $$\fcal_{\delta}^p =\left \{(s_1,\dots,s_m:  \left(
\sum |s_j(p)|^2_{h^N}\right)^{1/2}<\de\right\}\,.$$

Thus, the  difference between the Gaussian conditional density and
the pair correlation density  corresponds to the difference
between the family of systems $\fcal_{\delta}^p $ and
the family of systems $U_{\epsilon}^p $.
This comparison of the pair correlation
density and the Gaussian conditional density shows that in a
probabilistic sense,  the conditions `$\sss(p)$ is small' and
`$\s$ has a zero near $p$' are mutually singular.

\subsection{Comparison of the conditional expectation and pair correlation in codimension 1}\label{compare1}

We take a different approach
to comparing $K_1^N(z | p) $ and $K_{2 1}^N(z, p)$:
The scaling asymptotics of  $K_1^N(z | p) $ and $K_{2 1}^N(z, p)$
 are both given by  universal expressions in
 the  normalized \szego
kernel or two-point function $P_N(z,w)$ (defined in  (\ref{PN})).
This is to be expected since the two-point function is the only
invariant of a Gaussian random field. Indeed,  Proposition \ref{cond} 
says that
\begin{equation} \label{RELCOND}  K_1^N(z|p)
= \E_N\big(Z_s\big) + \frac 1{2\pi} (i\ddbar)_z\; Y (- \log P_N(z,p))\,,
\end{equation} where $Y(\la)=   \log (1-e^{-2\la})$ (recall \eqref{Y}).

We now review the approach to the pair correlation current $K^N_{21}(z,p)$ given in \cite{SZa}.  The  {\it pair correlation current} of zeros
$Z_s$ is given by
$\E_N\big(Z_s \boxtimes Z_s
\big)$, and the  {\it variance current\/}
is given by
\begin{equation}\label{vc}
\Var_N\big(Z_s\big): =
\E_N\big(Z_s \boxtimes Z_s \big)
- \E_N\big(Z_s\big)\boxtimes \E_N
\big(Z_s\big)\in \dcal'^{2k,2k}(M\times M).
\end{equation}
Here we  write
$$S\boxtimes T = \pi_1^*S \wedge \pi_2^*T \in \dcal'^{p+q}(M\times
M)\;, \qquad \mbox{for }\ S\in \dcal'^p(M),\ T\in \dcal'^q(M)\;,$$
where $\pi_1,\pi_2:M\times M\to M$ are the projections to the
first and second factors, respectively.

In \cite{SZa}, the first two authors  gave a  {\it
pluri-bipotential\/} for the variance current in codimension one,
i.e.\  a function $Q_N\in L^1(M\times M)$ such that
\begin{equation}\label{varcur}{\bf Var}_N\big(Z_{s}\big)=
(i\ddbar)_z\,(i\ddbar)_w \,Q_N(z,w) \;.\end{equation} 
The bipotential   $Q_N:M\times M\to
[0,+\infty)$ is given by
\begin{equation}
\label{qN} Q_N(z,w)= \wt G(P_N(z,w))\,,\quad \wt G(t) = -\frac 1{4\pi^2}
\int_0^{t^2} \frac{\log(1-s)}{s}\,ds\;. \end{equation}

The analogue to \eqref{RELCOND} for the pair
correlation current can be written \begin{equation}\label{BIPOT1}
K_{21}^N(z,p) =\E_N\big(Z_{s} \boxtimes Z_{s} \big)
= \E_N\big(Z_{s}\big)\boxtimes \E_N \big(Z_{s}\big) + \ddbar_z
\ddbar_p  F (- \log P_N(z,p)),
\end{equation}
where $F$ is the anti-derivative of the function $\frac 1{2\pi^2}Y$:
\begin{equation}\label{Gtilde1} F(\lambda) = \wt G(e^{-\la}) = -\frac
1{2\pi^2} \int_\la^\infty \log(1-e^{-2s})\,ds\;,\qquad \la\ge
0\;\end{equation} 
That is, $\frac 1{2\pi^2}Y (- \log P_N(z,p))$
is the relative potential between the conditioned and
unconditioned distribution of zeros, while  $ F (- \log P_N(z,p))$
is the relative {\it bi-potential} for the pair correlation
current $ \E_N\big(Z_{s} \boxtimes Z_{s} \big) $.

\end{document}